\documentclass[a4paper,11pt]{scrartcl}

\usepackage[T1]{fontenc}
\usepackage[latin1]{inputenc}
\usepackage[english]{babel}
\usepackage{setspace}
\usepackage{amsmath,amssymb,amsthm}
\usepackage{accents}
\usepackage{mathtools}
\usepackage{dsfont}
\usepackage{url}
\usepackage[section]{algorithm}
\usepackage{algpseudocode}
\usepackage{relsize}
\usepackage{array}
\usepackage{dcolumn}
\usepackage{sidecap}
\usepackage{draftwatermark}
\usepackage{sidecap}

\SetWatermarkText{PREPRINT}
\SetWatermarkLightness{0.85}
\SetWatermarkScale{4}

\setkomafont{title}{\normalfont\bfseries}
\setkomafont{section}{\normalfont\bfseries}
\setkomafont{subsection}{\normalfont\itshape}

\theoremstyle{plain}
\newtheorem{theorem}{Theorem}[section]
\newtheorem{lemma}[theorem]{Lemma}
\newtheorem{remark}[theorem]{Remark}
\newtheorem{definition}[theorem]{Definition}
\newtheorem{corollary}[theorem]{Corollary}

\title{Error control for the localized reduced basis multi-scale method with adaptive on-line enrichment}
\author{M. Ohlberger\thanks{Mario Ohlberger, Applied Mathematics, University of M\"unster, Einsteinstr. 62, D-48149 M\"unster, Germany, \texttt{mario.ohlberger@uni-muenster.de}} \and F. Schindler\thanks{Felix Schindler (formerly Albrecht), Applied Mathematics, University of M\"unster, Einsteinstr. 62, D-48149 M\"unster, Germany, \texttt{felix.schindler@uni-muenster.de}. This work has been supported by the German Federal Ministry of Education and Research (BMBF) under contract number 05M13PMA.}}

\newcommand{\hnS}{\hspace{-1.25pt}}
\newcommand{\divergence}{\nabla\hnS\hnS\cdot\hnS}
\newcommand{\gradient}{\nabla\hnS}
\newcommand{\gradienth}{\nabla\hnS_h}
\newcommand{\boundary}[1]{\ensuremath{\partial\hnS #1}}
\newcommand{\restrict}[2]{{\ensuremath{\left. #1\right|_{#2}}}}
\newcommand{\restrictInline}[2]{{\ensuremath{#1\big|_{#2}}}}
\newcommand{\komma}{\text{,}}
\newcommand{\punkt}{\text{.}}
\newcommand{\mydot}{\hnS\cdot\hnS}
\newcommand{\N}{\mathbb{N}}
\newcommand{\R}{\mathbb{R}}
\newcommand{\Order}{\mathcal{O}}
\newcommand{\Pk}{\mathbb{P}}
\newcommand{\Hdiv}{H_\text{div}}
\newcommand{\param}{\boldsymbol{\mu}}
\newcommand{\Params}{\mathcal{P}}
\newcommand{\paramFixed}{\overline{\param}}
\newcommand{\paramHat}{\hat{\param}}
\newcommand{\eps}{\varepsilon}
\newcommand{\alphaParam}{\alpha(\param, \paramFixed)}
\newcommand{\alphaParamHat}{\alpha(\param, \paramHat)}
\newcommand{\gammaParam}{\gamma(\param, \paramFixed)}
\newcommand{\gammaParamHat}{\gamma(\param, \paramHat)}
\newcommand{\red}{\ensuremath{\text{\textnormal{red}}}}
\newcommand{\overcirc}[1]{\accentset{\circ}{#1}}
\newcommand{\Triangulation}{\mathcal{T}_H}
\newcommand{\triangulation}{\tau_h}
\newcommand{\tildetriangulation}{\tilde{\tau}_h}
\newcommand{\tildeTriangulation}{\tilde{\mathcal{T}}_H}
\newcommand{\Element}{T}
\newcommand{\element}{t}
\newcommand{\vertex}{\nu}
\newcommand{\neighbor}{s}
\newcommand{\Neighbor}{S}
\newcommand{\Neighbors}{\mathcal{N}}
\newcommand{\Faces}{\mathcal{F}_H}
\newcommand{\faces}{\mathcal{F}_h}
\newcommand{\tildefaces}{\tilde{\mathcal{F}}_h}
\newcommand{\innerfaces}[1]{\overcirc{\mathcal{F}}_#1}
\newcommand{\boundaryfaces}[1]{\overline{\mathcal{F}}_#1}
\newcommand{\face}{e}
\newcommand{\Face}{E}
\newcommand{\mean}[1]{\ensuremath{\left\{\hnS\hnS\left\{#1\right\}\hnS\hnS\right\}}}
\newcommand{\jump}[1]{\ensuremath{\left[\hnS\left[#1\right]\hnS\right]}}
\newcommand{\energynorm}[2]{\ensuremath{{\left|\hnS\left|\hnS\left|#1\right|\hnS\right|\hnS\right|}_{#2}}}
\newcommand{\jumpnorm}[2]{\ensuremath{{\left[\hnS\left[\hnS\left[#1\right]\hnS\right]\hnS\right]}_{#2}}}
\newcommand{\norm}[2]{\ensuremath{{\left|\hnS\left|#1\right|\hnS\right|}_{#2}}}
\newcommand{\one}{\mathds{1}}
\DeclareMathOperator{\dx}{dx}
\DeclareMathOperator{\ds}{ds}

\newcommand{\dune}[1]{\texttt{dune-#1}}
\newcommand\Cpp{C\nolinebreak[4]\hspace{-.05em}\raisebox{.4ex}{\relsize{-3}{\textbf{++}}}}
\newcolumntype{s}{D{,}{\cdot}{-1}}
\newcommand{\sci}[2]{#1,10^{#2}}

\newcommand{\mcol}[2]{\multicolumn{1}{#1}{$#2$}}

\definecolor{fiveclassYlGnBu1}{rgb}{1.00,1.00,0.80}
\definecolor{fiveclassYlGnBu2}{rgb}{0.63,0.85,0.71}
\definecolor{fiveclassYlGnBu3}{rgb}{0.25,0.71,0.77}
\definecolor{fiveclassYlGnBu4}{rgb}{0.17,0.50,0.72}
\definecolor{fiveclassYlGnBu5}{rgb}{0.15,0.20,0.58}

\definecolor{fiveclassRdYlBu1}{rgb}{0.84,0.10,0.11}
\definecolor{fiveclassRdYlBu2}{rgb}{0.99,0.68,0.38}
\definecolor{fiveclassRdYlBu3}{rgb}{1.00,1.00,0.75}
\definecolor{fiveclassRdYlBu4}{rgb}{0.67,0.85,0.91}
\definecolor{fiveclassRdYlBu5}{rgb}{0.17,0.48,0.71}

\definecolor{sixclassRdYlBu1}{rgb}{0.84,0.19,0.15}
\definecolor{sixclassRdYlBu2}{rgb}{0.99,0.55,0.35}
\definecolor{sixclassRdYlBu3}{rgb}{1.0,0.88,0.56}
\definecolor{sixclassRdYlBu4}{rgb}{0.88,0.95,0.97}
\definecolor{sixclassRdYlBu5}{rgb}{0.57,0.75,0.86}
\definecolor{sixclassRdYlBu6}{rgb}{0.27,0.46,0.71}

\begin{document}

\maketitle

\begin{small}
  \textbf{Abstract.}
  In this contribution we consider localized, robust and efficient a-posteriori error estimation of the \underline{l}ocalized \underline{r}educed \underline{b}asis \underline{m}ulti-\underline{s}cale (LRBMS) method for parametric elliptic problems with possibly heterogeneous diffusion coefficient.
The numerical treatment of such parametric multi-scale problems are characterized by a high computational complexity, arising from the multi-scale character of the underlying differential equation and the additional parameter dependence.
The LRBMS method can be seen as a combination of numerical multi-scale methods and model reduction using \underline{r}educed \underline{b}asis (RB) methods to efficiently reduce the computational complexity with respect to the multi-scale as well as the parametric aspect of the problem, simultaneously.
In contrast to the classical residual based error estimators currently used in RB methods, we are considering error estimators that are based on conservative flux reconstruction and provide an efficient and rigorous bound on the full error with respect to the weak solution.
In addition, the resulting error estimator is localized and can thus be used in the on-line phase to adaptively enrich the solution space locally where needed.
The resulting certified LRBMS method with adaptive on-line enrichment thus guarantees the quality of the reduced solution during the on-line phase, given any (possibly insufficient) reduced basis that was generated during the offline phase.
Numerical experiments are given to demonstrate the applicability of the resulting algorithm with online enrichment to single phase flow in heterogeneous media.

\end{small}

\vspace{1ex}

\textbf{Key words.} model reduction, multi-scale, error analysis, flux reconstruction.

\vspace{1ex}

\textbf{AMS subject classifications.} 65G99, 65N55, 65N15, 35J20, 65N30, 76S05.

\section{Introduction}

We are interested in efficient and reliable numerical approximations of elliptic parametric multi-scale problems.
Such problems consist of finding $p(\param) \in Q$, such that
\begin{align}
  b(p(\param), q; \param) = l(q)
    &&\text{for all } q \in Q\komma
  \label{equation::introduction::general_variational_setting}
\end{align}
in a suitable space $Q$ for parameters $\param\in\Params \subset \R^p$, for $p \in \N$, where the data functions involved may depend on an a-priori given multi-scale parameter $\eps > 0$ (described in detail in Section \ref{section::problem_formulation}).
An approximation $p_h(\param) \in Q_h$ of $p(\param)$ is usually obtained by a discretization of \eqref{equation::introduction::general_variational_setting} resulting from a Galerkin projection onto a discrete space $Q_h$, associated with a triangulation $\triangulation$ of $\Omega$.
Since $\triangulation$ should resolve all features of \eqref{equation::introduction::general_variational_setting} associated with the fine scale $\eps$, solving parametric heterogeneous multi-scale problems accurately can be challenging and computationally costly, in particular for strongly varying scales and parameter ranges (see the references in \cite{Ohl2012}).
Two traditional approaches exist to reduce the computational complexity of the discrete problem: numerical multi-scale methods and model order reduction techniques.
Numerical multi-scale methods reduce the complexity of multi-scale problems with respect to $\eps$ for a fixed $\param$, while model order reduction techniques reduce the complexity of parametric problems with respect to $\param$ for moderate scales $\eps$ (see \cite{Ohl2012} for an overview).
In general, numerical multi-scale methods capture the macroscopic behavior of the solution in a coarse approximation space $Q_H \subset Q_h$, e.g., associated with a coarse triangulation $\Triangulation$ of $\Omega$, and recover the microscopic behavior of the solution by local fine-scale corrections.
Inserting this additive decomposition into \eqref{equation::introduction::general_variational_setting} yields a coupled system of a fine- and a coarse-scale variational problem.
By appropriately selecting trial and test spaces and defining the localization operators to decouple this system, a variety of numerical multi-scale methods can be recovered, e.g., the \underline{m}ulti-\underline{s}cale \underline{f}inite \underline{e}lement \underline{m}ethod (MsFEM) \cite{EH2009}, the variational multi-scale method \cite{Nor2008,LM2007}, the multi-scale finite volume method \cite{HJ2011} and the \underline{h}eterogeneous \underline{m}ulti-scale \underline{m}ethod (HMM) \cite{AEEV2012,HO2009}, just to name a few.
Model order reduction using \underline{r}educed \underline{b}asis (RB) methods, on the other hand, is based on the idea to introduce a reduced space $Q_\red \subset Q_h$, spanned by discrete solutions for a limited number of parameters $\param$.
These training parameters are iteratively selected by an adaptive Greedy procedure (see e.g. \cite{VPP2003} and the references therein).
Depending on the choice of the training parameters and the nature of the problem, $Q_\red$ is expected to be of a significantly smaller dimension than $Q_h$.
Additionally, if $b$ allows for an affine decomposition with respect to $\param$, its components can be projected onto $Q_\red$, which can then be used to effectively split the computation into an off-line and on-line part.
In the off-line phase all parameter-independent quantities are precomputed, such that the on-line phase's complexity only depends on the dimension of $Q_\red$.
The idea of the recently presented \underline{l}ocalized \underline{r}educed \underline{b}asis \underline{m}ulti-\underline{s}cale (LRBMS) approach (see \cite{KOH2011,AHKO2012}) is to combine numerical multi-scale and RB methods and to generate a local reduced space $Q_\red^\Element \subset Q_h^\Element$ for each coarse element of $\Triangulation$, given a tensor product type decomposition of the fine approximation space, $Q_h = \oplus_{\Element\in\Triangulation} Q_h^\Element$.
The coarse reduced space is then given as $Q_\red(\Triangulation) := \oplus_{\Element\in\Triangulation} Q_\red^\Element \subset Q_h$, resulting in a multiplicative decomposition of the solution into $p_\red(x; \param)=\sum_{n=1}^{\dim Q_\red(\Triangulation)} p_n(x; \param) \varphi_n(x)$, where the RB functions $\varphi_n$ capture the microscopic behavior of the solution associated with the fine scale $\eps$ and the coefficient functions $p_n$ only vary on the coarse partition $\Triangulation$.

Other model reduction approaches for parametric (but not multi-scale) problems that incorporate localization ideas and concepts from \underline{d}omain \underline{d}ecomposition (DD) methods are the reduced basis element method \cite{MR2002,LMR2007}, the reduced basis hybrid method \cite{IQR2012,IQRV2014} and the port reduced static condensation reduced basis element method \cite{Sme2015}.
While the idea of the former is to share one reduced basis on all subdomains the idea of the latter two is to generate one reduced basis for each class of subdomains which are then coupled appropriately.
There also exist several approaches to use model reduction techniques for homogenization problems (see \cite{Boy2008}) and problems with multiple scales, such as the reduced basis finite element HMM \cite{AB2013,AB2014}.
For the case of no scale-separation there exist the generalized MsFEM method \cite{EGH2013}, which incorporates model reduction ideas, and most recently a work combining the reduced basis method with localized orthogonal decomposition (see \cite{AH2014a}).

It is vital for an efficient and reliable use of RB as well as LRBMS methods to have access to an estimate on the model reduction  error.
Such an estimate is used to drive the adaptive Greedy basis generation during the off-line phase of the computation and to ensure the quality of the reduced solution during the on-line phase.
It is usually given by a residual based estimator involving the stability constant and the residual in a dual norm.
It was shown in \cite{AHKO2012} that such an estimator can be successfully applied in the context of the LRBMS, but it was also pointed out that an estimator relying on global information might not be computationally feasible since too much work is required in the off-line part of the computation.

The novelty of this contribution lies in a completely different approach to error estimation -- at least in the context of RB methods.
We make use of the ansatz of local error estimation presented in \cite{ESV2010} which measures the error by a conforming reconstruction of the physical quantities involved, specifically the diffusive flux $-\lambda(\param)\kappa_\eps \gradient p(\param)$.
This kind of local error estimation was proven to be very successful in the context of multi-scale problems and very robust with respect to $\eps$.
We show in this work how we can transfer those ideas to the framework of the LRBMS to obtain an estimate of the error
$\energynorm{p(\param) - p_\red(\param)}{\param}$.
We would like to point out that we are able to estimate the error against the weak solution $p(\param)$ in a parameter dependent energy norm while traditional RB approaches only allow to estimate the model reduction error in a parameter independent norm and only against the discrete solution.
In principal, this approach is able to turn the LRBMS method into a full multi-scale approximation scheme, while traditional RB methods can only be seen as a model reduction technique.
We would also like to point out that, to the best of our knowledge, this approach (first published in \cite{OS2014}) is the first one to make use of local error information in the context of RB methods.

\section{Problem formulation and discretization}
\label{section::problem_formulation}

We consider linear elliptic problems in a bounded connected domain $\Omega \subset \R^d$, for $d=2, 3$, with polygonal boundary $\boundary{\Omega}$ and a multiplicative splitting of the influences of the multi-scale parameter $\eps$ and the parameter $\param$.
An example is given by the problem of finding a global pressure $p(\param) \in H^1_0(\Omega)$ for a set of admissible parameters $\param \in \Params$, such that
\begin{align}
  - \divergence \big(
        \lambda(\param) \kappa_\eps \gradient p(\param)
      \big)
    = f
    &&\text{in } \Omega
  \label{equation::problem::global_pressure}
\end{align}
in a weak sense in $H^1_0(\Omega)$, where $H^1$ denotes the usual Sobolev space of weakly differentiable functions and $H^1_0(\Omega) \subset H^1(\Omega)$ its elements which vanish on the boundary in the sense of traces.
Problems of this kind arise for instance in the context of the global pressure formulation of two-phase flow in porous media.
Using an IMPES discretization scheme of the pressure/saturation system (see \cite{CHM2006} and the references therein), an equation of the form \eqref{equation::problem::global_pressure} has to be solved in each time step for a different parameter $\param \in \Params$.
In that context $\lambda(\param)$ denotes the scalar total mobility and $\kappa_\eps$ denotes a possibly complex heterogeneous permeability tensor; external forces are collected in $f$ and the parameter $\param$ models the influence of the global saturation (see Definition \ref{definition::problem::weak_solution} for details on $\lambda$, $\kappa_\eps$ and $f$).
Note that more complex boundary conditions and additional parameter dependencies of the boundary values as well as the right hand side (modeling parameter dependent wells, for instance) are possible, but do not lie within the scope of this work.

\textbf{Triangulation.} We require two nested partitions of $\Omega$, a coarse one, $\Triangulation$, and a fine one, $\triangulation$.
Let $\triangulation$ be a simplicial triangulation of $\Omega$ with elements $\element \in \triangulation$.
In the context of multi-scale problems we call $\tau_h$ a \emph{fine triangulation} if it resolves all features of the quantities involved in \eqref{equation::problem::global_pressure}, specifically if $\restrict{\kappa_\eps}{\element}\in [L^\infty(\element)]^{d\times d}$ is constant for all $\element \in \triangulation$.
This assumption is a natural one in the context of two-phase flow problems where the permeability field is usually given by piecewise constant measurement data.
For simplicity we require $\triangulation$ to fulfill the requirements stated in \cite[Sect. 2.1]{ESV2010}, namely shape-regularity and the absence of hanging nodes; an extension to more general triangulations is possible analogously to \cite[A.1]{ESV2010}.
We only require the coarse elements $\Element\in\Triangulation$ to be shaped such that a local Poincar\'{e} inequality in $H^1(\Element)$ is fulfilled (see the requirements of Theorem \ref{theorem::error::locally_computable_abstract_energy_norm_estimate}).
We collect all fine faces in $\faces$, all coarse faces in $\Faces$ and denote by $\Neighbors(\element) \subset \triangulation$ and $\Neighbors(\Element) \subset \Triangulation$ the neighbors of $\element \in \triangulation$ and $\Element \in \Triangulation$, respectively and by $h_*$ the diameter of any element $*$ of the sets $\triangulation$, $\Triangulation$, $\faces$ or $\Faces$.
We collect in $\triangulation^\Element \subset \triangulation$ the fine elements of $\triangulation$ that cover the coarse element $\Element$ and in $\faces^* \subset \faces$ all faces that cover the set $*$, e.g. by $\faces^\element$ the faces of a fine element $\element \in \triangulation$, by $\faces^\Face$ the faces that cover a coarse face $\Face \in \Faces$ and so forth; the same notation is used for coarse faces $\Faces^* \subset \Faces$.
In addition we denote the set of all boundary faces by $\boundaryfaces{h} \subset \faces$ and the set of all inner faces, that share two elements, by $\innerfaces{h} \subset \faces$, such that $\boundaryfaces{h} \cup \innerfaces{h} = \faces$ and $\boundaryfaces{h} \cap \innerfaces{h} = \emptyset$.
We also denote the set of fine faces which lie on the boundary of any coarse element $\Element \in \Triangulation$ by $\boundaryfaces{h}^\Element := \bigcup_{\Face \in \Faces^\Element} \faces^\Face$ and by $\innerfaces{h}^\Element := \faces^\Element \backslash \boundaryfaces{h}^\Element$ the set of fine faces which lie in the interior of the coarse element.
Finally, we assign a unit normal $n_\face \in \R^d$ to each inner face $\boundary \element^- \cap \boundary \element^+ = \face \in \innerfaces{h}$, pointing from $\element^-$ to $\element^+$, and also denote the unit outward normal to $\boundary \Omega$ by $n_\face$ for a boundary face $\face = \boundary\element^- \cap \boundary{\Omega}$, for $\element^-\text{,} \element^+ \in \triangulation$.

\subsection{The continuous problem}

For our analysis we define the \emph{broken Sobolev space} $H^1(\triangulation) \subset L^2(\Omega)$ by $H^1(\triangulation) := \big\{ q \in L^2(\Omega) \;\big|\; q|_\element \in H^1(\element) \;\; \forall \element\in\triangulation \big\}$ as the most general space for the weak, the discrete and the reduced solution defined below.
In the same manner we denote the local broken Sobolev spaces $H^1(\triangulation^\Element)\subset L^2(\Element)$ for all coarse elements $\Element\in\Triangulation$ and denote by $\gradienth:H^1(\triangulation)\to[L^2(\Omega)]^d$ the \emph{broken gradient operator} which is locally defined by $\restrict{(\gradienth q)}{\element}:=\gradient (\restrict{q}{\element})$ for all $\element\in\triangulation$ and $q\in H^1(\triangulation)$.
Now given $\lambda(\param)$, $\kappa_\eps$ and $f$ as stated in Definition \ref{definition::problem::weak_solution} we define the parametric bilinear form $b: \Params \to [H^1(\triangulation) \times H^1(\triangulation) \to \R]$, $\param \mapsto [(p, q) \mapsto b(p, q; \param)]$, and the linear form $l: H^1(\triangulation) \to \R$ by $b(p, q; \param) := \sum_{\Element\in\Triangulation} b^T(p, q; \param)$ and $l(q) := \sum_{\Element\in\Triangulation} l^\Element(q)$, respectively, and their local counterparts $b^\Element : \Params \to [H^1(\triangulation^\Element) \times H^1(\triangulation^\Element) \to \R]$, $\param \mapsto [(p, q) \mapsto b^\Element(p, q; \param)]$, and $l^\Element: H^1(\triangulation^\Element) \to \R$ for all $\Element\in\Triangulation$, $\param\in\Params$ and $p\komma q \in H^1(\triangulation)$ by
\begin{align}
  b^\Element(p, q; \param) := \int\limits_\Element (\lambda(\param) \kappa_\eps \gradienth p)\mydot\gradienth q \dx
    &&\text{and}&&
  l^\Element(q):= \int\limits_\Element f q\dx \komma
  \notag
\end{align}
respectively.

\begin{definition}[Weak solution]
\label{definition::problem::weak_solution}
  Let $f \in L^2(\Omega)$ be bounded, $\lambda(\param) \in L^\infty(\Omega)$ be strictly positive for all $\param\in\Params$ and $\kappa_\eps \in [L^\infty(\Omega)]^{d\times d}$ symmetric and uniformly positive definite, such that $\lambda(\param) \kappa_\eps \in [L^\infty(\Omega)]^{d\times d}$ is bounded from below (away from 0) and above for all $\param\in\Params$.
  We define the \emph{weak solution} $p: \Params \to H^1_0(\Omega)$ for a parameter $\param\in\Params$, such that
  \begin{align}
    b(p(\param), q; \param) = l(q) &&\text{for all } q \in H^1_0(\Omega)\punkt
    \label{equation::problem::weak_solution}
  \end{align}
\end{definition}

Note that, since $b(\cdot, \cdot; \param)$ is continuous and coercive for all $\param\in\Params$ (due to the assumptions on $\lambda(\param) \kappa_\eps$) and since $l$ is bounded, there exists a unique solution of \eqref{equation::problem::weak_solution} due to the Lax-Milgram Theorem.
Given these requirements we denote by $0 < c_\eps^\element(\param)$ and $c_\eps^\element(\param) \leq C_\eps^\element(\param)$ the smallest and largest eigenvalue of $\restrict{(\lambda(\param) \kappa_\eps)}{\element}$, respectively, for any $\param \in \Params$ and $\element \in \triangulation$ and additionally define $0 < c_\eps^\element := \min_{\param\in\Params} c_\face^\element(\param)$ and $c_\face^\element < C_\face^\element := \max_{\param\in\Params} C_\face^\element(\param)$.

\textbf{A note on parameters.} In addition to the assumptions we posed on $\lambda$ above we also demand it to be \emph{affinely decomposable} with respect to $\param\in\Params$, i.e there exist $\varXi \in \N$ strictly positive \emph{coefficients} $\theta_\xi:\Params\to\R$ and $\varXi$ nonparametric \emph{components} $\lambda_\xi \in L^\infty(\Omega)$, for $1 \leq \xi \leq \varXi$, such that $\lambda(x; \param) = \sum_{\xi=1}^{\varXi} \theta_\xi(\param) \lambda_\xi(x)$.
We can then compare $\lambda$ for two parameters $\param, \paramFixed \in \Params$ by $\alphaParam \lambda(\paramFixed) \;\leq\; \lambda(\param) \;\leq\; \gammaParam \lambda(\paramFixed)$, where $\alphaParam :=\min_{\xi = 1}^{\varXi} {\theta_\xi(\param)} {\theta_\xi(\paramFixed)}^{-1}$ and $\gammaParam := \max_{\xi = 1}^{\varXi} {\theta_\xi(\param)}{\theta_\xi(\paramFixed)}^{-1}$
denote the positive equivalence constants.
This assumption on the data function $\lambda$ is a common one in the context of RB methods and covers a wide range of physical problems.
If $\lambda$ does not exhibit such a decomposition one can replace $\lambda$ by an arbitrary close approximation using Empirical Interpolation techniques (see \cite{BMNP2004,DHO2012}) which does not impact our analysis.
All quantities that linearly depend on $\lambda$ inherit the above affine decomposition in a straightforward way.
In particular there exist component bilinear forms such that $b$ and $b^\Element$ (and their discrete counterparts introduced further below) are also affinely decomposable.
Since we would like to estimate the error in a problem dependent norm we introduce a \emph{parameter dependent energy (semi) norm} $\energynorm{\cdot}{\cdot}: \Params \to [H^1(\triangulation)\to\R]$, $\param \to [q \to \energynorm{q}{\param}]$, defined by $\energynorm{q}{\param} := \big(\sum_{\Element\in\Triangulation} \energynorm{q}{\param, \Element}^2 \big)^{1/2}$ with the local semi-norm defined by $\energynorm{q}{\param, \Element} := b^\Element(q, q; \param)^{1/2}$, for all $\Element\in\Triangulation$, $\param \in \Params$ and $q \in H^1(\triangulation)$; the local semi-norm on the fine triangulation, $\energynorm{\cdot}{\param, \element}$ for any $\element \in \triangulation$, is defined analogously to $\energynorm{\cdot}{\param, \Element}$.
Note that $\energynorm{\cdot}{\cdot}$ is a norm only on $H^1_0(\Omega)$.
We can compare these norms for any two parameters $\param, \paramFixed \in \Params$ using the above decomposition of $\lambda$ (the same holds true for the local semi-norms):
\begin{align}
  \sqrt{\alphaParam} \;\; \energynorm{\cdot}{\paramFixed}
    \;\;\;\leq\;\;\; \energynorm{\cdot}{\param}
    \;\;\;\leq\;\;\; \sqrt{\gammaParam} \;\; \energynorm{\cdot}{\paramFixed}
  \label{equation::problem::parameter_norm_equivalence}
\end{align}

\subsection{The discretization}
\label{subsection::discretization}

We discretize \eqref{equation::problem::weak_solution} by allowing for a suitable discretization of at least first order inside each coarse element $\Element\in\Triangulation$ and by coupling those with a \underline{S}ymmetric \underline{w}eighted \underline{I}nterior-\underline{P}enalty \underline{d}iscontinuous \underline{G}alerkin (SWIPDG) discretization along the coarse faces of $\Triangulation$.
This ansatz can be either interpreted as an extension of the SWIPDG discretization on the coarse partition $\Triangulation$, where we further refine each coarse element and introduce an additional local discretization, or it can be interpreted as a domain-decomposition approach, where we use local discretizations, defined on subdomains given by the coarse partition, which are then coupled by the SWIPDG fluxes.
In view of the latter, this ansatz shows some similarities to \cite{BZ2006} but allows for a wider range of local discretizations.
A similar ansatz for a multi-numerics discretization using a different coupling strategy was independently developed and recently introduced in \cite{PVWW2013}.
We will present two particular choices for the local discretizations and continue with the definition of the overall DG discretization.

\textbf{Local discretizations.} The main idea is to approximate the local bilinear forms $b^\Element$, which are defined on the local subdomain triangulations $\triangulation^\Element$, by local discrete bilinear forms $b_h^\Element$ discretizing \eqref{equation::problem::global_pressure} on $\Element$ with homogeneous Neumann boundary values.
We will additionally choose local discrete ansatz spaces $Q_h^{k, \Element} \subset H^1(\triangulation^\Element)$, with polynomial order $k \in \N$, to complete the definition of the local discretizations.
The first natural choice for the local discretization is to use a standard \underline{c}ontinuous \underline{G}alerkin (CG) discretization, which we obtain by setting $b_h^\Element = b^\Element$ and $Q_h^{k, \Element} = S_h^k(\triangulation^\Element)$ with
\begin{align}
  S_h^k(\triangulation^\Element) &:=
      \big\{
        q \in C^0(\Element)
      \;\big|\;
        \restrict{q}{\element} \in \Pk_k(\element) \quad \forall \element \in \triangulation^\Element
      \big\}
    \;\;\subset\;\; H^1(\Element) \komma
  \notag
\intertext{%
  where $\Pk_k(\omega)$ denotes the set of polynomials on $\omega \subseteq \Omega$ with total degree at most $k\in \N$.
  Another choice is to use a discontinuous space $Q_h^{k, \Element} = Q_h^k(\triangulation^\Element)$, given by
}
  Q_h^k(\triangulation^\Element) &:=
      \big\{
        q \in L^2(\Element)
      \;\big|\;
        \restrict{q}{\element} \in \Pk_k(\element) \quad \forall \element \in \triangulation^\Element
      \big\}
    \;\;\subset\;\; H^1(\triangulation^\Element) \komma
  \notag
\end{align}
and to choose $b_h^\Element$ from a family of DG discretizations.
Therefore we introduce the technicalities needed to state a common framework for the non symmetric, the incomplete, the symmetric and the symmetric weighted \underline{i}nterior-\underline{p}enalty (IP) DG discretization (henceforth denoted by NIPDG, IIPDG, SIPDG and SWIPDG, respectively, see \cite{ESZ2009} and the references therein), following \cite[Sect. 2.3]{ESV2010}.

For a function $q \in H^1(\triangulation)$ that is double-valued on interior faces we denote its jump on an inner face $\face \in \innerfaces{h}$ by $\jump{q}_\face := q^- - q^+$ with $q^\pm := \restrict{q}{\element^\pm}$, recalling that $\face = \element^- \cap \element^+$ for $\element^\pm \in \triangulation$.
We also assign weights $\omega_\face^-\text{, } \omega_\face^+ > 0$ to each inner face, such that $\omega_\face^- + \omega_\face^+ = 1$, and denote the weighted average of $q$ by $\mean{q}_\face := \omega_\face^- q^- + \omega_\face^+ q^+$.
On a boundary face $\face \in \boundaryfaces{h}$ we set $\omega_\face^- = 1$, $\omega_\face^+ = 0$, $\jump{q}_\face := q$ and $\mean{q}_\face := q$.
With these definitions we define the local bilinear form $b_h^\Element: \Params \to [H^1(\triangulation^\Element) \times H^1(\triangulation^\Element) \to \R]$, $\param \mapsto [(p, q) \mapsto b_h^\Element(p, q; \param)]$ for $\vartheta \in \{-1, 0, 1\}$ by
\begin{align}
  b_h^\Element (p, q; \param) := b^\Element(p,q; \param)
    +\sum_{\face \in \innerfaces{h}^\Element}
      \Big( \vartheta b_c^\face(q, p; \param) + b_c^\face(p, q; \param) + b_p^\face(q, p; \param) \Big)
  \label{equation::problem::local_generalized_IPDG_bilinearform}
\end{align}
on $\Element \in \Triangulation$, with its coupling and penalty parts $b_c^\face$ and $b_p^\face$, respectively, defined by
\begin{align}
  b_c^\face(p, q; \param) := \int\limits_\face - \mean{(\lambda(\param) \kappa_\eps \gradient p) \mydot n_\face}_\face \jump{q}_\face \ds
  &&\text{and}&&
    b_p^\face(p, q; \param) := \int\limits_\face \sigma_\face(\param) \jump{p}_\face \jump{q}_\face \ds
    \komma
  \notag
\end{align}
for all $\param \in \Params$, all $p, q \in H^1(\triangulation)$ and all $\face \in \faces$.
The parametric positive penalty function $\sigma_\face(\param): \Params \to \R$ is given by $\sigma_\face(\param) := \sigma h_\face^{-1} \mean{\lambda(\param)}_\face \sigma_\eps^\face$, where $\sigma \geq 1$ denotes a user-dependent parameter and the locally adaptive weight is given by $\sigma_\eps^\face := \delta_\face^+ \delta_\face^-(\delta_\face^+ + \delta_\face^-)^{-1}$ for an interior face $\face \in \innerfaces{h}$ and by $\sigma_\eps^\face := \delta_\face^-$ on a boundary face $\face \in \boundaryfaces{h}$, respectively, with $\delta_\face^\pm := n_\face \mydot \kappa_\eps^\pm \mydot n_\face$.
Using the weights $\omega_\face^\pm := 1/2$ we obtain the NIPDG bilinear form for $\vartheta = -1$, the IIPDG bilinear form for $\vartheta = 0$ and the SIPDG bilinear form for $\vartheta = 1$.
We obtain the SWIPDG bilinear form for $\vartheta = 1$ by using locally adaptive weights $\omega_\face^- := \delta_\face^+(\delta_\face^+ + \delta_\face^-)^{-1}$ and $\omega_\face^+ := \delta_\face^-(\delta_\face^+ + \delta_\face^-)^{-1}$.
From now on we assume that $b_h^\Element$ is of the form \eqref{equation::problem::local_generalized_IPDG_bilinearform}, since this is the most general case and also covers a CG discretization, where all face terms vanish due to the nature of $S_h^k(\triangulation^\Element)$.

\textbf{Global coupling.} Now given suitable local discretizations on the coarse elements $\Element \in \Triangulation$ we couple those along the coarse faces $\Face \in \Faces$ using a SWIPDG discretization again and define the bilinear form $b_h: \Params \to [H^1(\triangulation) \times H^1(\triangulation) \to \R]$, $\param \mapsto [(p, q) \mapsto b_h(p, q; \param)]$, by
\begin{align}
  b_h(p, q; \param) &:=
    \sum_{\Element \in \Triangulation}
      b_h^\Element(p, q; \param)
    +\sum_{\Face \in \Faces}
      b_h^\Face(p, q; \param)\komma
  \label{equation::problem::global_bilinear_form}
\intertext{%
  where we use the SWIPDG variants of $\omega_\face^\pm$ to define the coupling bilinear form $b_h^\Face: \Params \to [H^1(\triangulation) \times H^1(\triangulation) \to \R]$, $\param \to [(p, q) \to b_h^\Face(p, q; \param)]$, by
}
  b_h^\Face(p, q; \param) &:=
  \sum_{\face \in \faces^\Element}
      \Big( b_c^\face(q, p; \param) + b_c^\face(p, q; \param) + b_p^\face(q, p; \param) \Big)
\end{align}
for all $\Face \in \Faces$, all $\param \in \Params$ and all $p, q \in H^1(\triangulation)$.
Accordingly, we define the DG space $Q_h^k \subset H^1(\triangulation)$ for $k \geq 1$ by
\begin{align}
  Q_h^k :=
    \big\{
      q \in H^1(\triangulation)
    \;\;\big|\;\;
      \restrict{q}{\Element} \in Q_h^{k, \Element} \quad \forall \Element \in \Triangulation
    \big\}\komma
  \notag
\end{align}
with $Q_h^{k, \Element}$ either being the local CG space $S_h^k(\triangulation^\Element)$ or the local DG space $Q_h^k(\triangulation^\Element)$.

\begin{definition}[DG solution]
  Let $\sigma$ be large enough, such that $b_h(\cdot, \cdot; \param)$ is continuous and coercive for all $\param \in \Params$.
  We define the \emph{DG solution} $p_h: \Params \to Q_h^1$ for a parameter $\param \in \Params$, such that
  \begin{align}
    b_h(p_h(\param), q_h; \param) = l(q_h) &&\text{for all } q_h \in Q_h^1\punkt
    \label{equation::problem::discrete_solution}
  \end{align}
\end{definition}

As always with DG methods, coercivity is to be understood with respect to a DG norm on $Q_h^1$ (i.e., given by the semi $H^1$ norm combined with a DG jump norm) and there exists a unique solution of \eqref{equation::problem::discrete_solution} due to the Lax-Milgram Theorem, if $\sigma$ is chosen accordingly.

\begin{remark}[Properties of the discretization]
  Depending on the choice of the coarse partition and the local discretization we can recover several discretizations from \eqref{equation::problem::global_bilinear_form}.
  Choosing $\Triangulation = \Omega$ and $Q_h^{k, \Element} = S_h^k(\triangulation^\Element)$ yields a standard CG discretization (except for the treatment of the boundary values), for instance.
  Choosing $\Triangulation = \triangulation$, on the other hand, results in the standard SWIPDG discretization proposed in \cite{ESZ2009,ESV2010}.
  Note that the local discretizations as well as the polynomial degree $k$ do not have to coincide on each coarse element (or even inside a coarse element when using a local DG space).
  It is thus possible to balance the computational effort by choosing local CG or $k$-adaptive DG discretizations.
  This puts our discretization close to the multi-numerics discretization proposed in \cite{PVWW2013}, where the latter allows for an even wider range of local discretizations while coupling along the coarse faces using Mortar methods.
  Concerning the choice of the user dependent penalty factor, we found an automated choice of $\sigma$ depending on the polynomial degree $k$, as proposed in \cite{ER2007}, to work very well.
\end{remark}

\section{Model reduction}
\label{section::model_reduction}

Disregarding the multi-scale parameter $\eps$ for the moment, model reduction using \underline{r}educed \underline{b}asis (RB) methods is a well established technique to reduce the computational complexity of \eqref{equation::problem::discrete_solution} with respect to the parameter $\param$.
It has been successfully applied in multi-query contexts, where \eqref{equation::problem::discrete_solution} has to be solved for many parameters $\param$ (e.g. in the context of optimization), and in real-time contexts, where a solution of \eqref{equation::problem::discrete_solution} (or a quantity depending on it) has to be available for some $\param$ as fast as possible (see \cite{AHKO2012,DHO2012,Ohl2012} and the references therein).
The general idea of RB methods is to span a reduced space $Q_\red \subset Q_h^k$ by post-processed solutions of \eqref{equation::problem::discrete_solution} for a selection of parameters and to apply a Galerkin projection of \eqref{equation::problem::discrete_solution} onto this reduced space.
Together with the assumption we posed on the data function in section \ref{section::problem_formulation} this allows for the well-known off-line/on-line decomposition of the RB method, where all quantities depending on $Q_h^k$ (in particular high-dimensional evaluations of the bilinear form) can be precomputed and stored in so called Gramian-matrices during a possibly computationally expensive off-line phase.
In the on-line phase of the simulation only low dimensional quantities depending on the parameter dependency of the data functions and the dimension of $Q_\red$ need to be computed which can usually be done in real time and even on low end devices such as smartphones or embedded devices.

It has been shown in \cite{AHKO2012}, however, that the computational complexity of the off-line phase can become unbearable if the RB method is applied to parametric \emph{multi-scale} problems, such as \eqref{equation::problem::discrete_solution} for small scales $\eps \ll 1$.
There are two main drawbacks of classical RB methods in this context: $(i)$ all high-dimensional quantities (solution snapshots, operator matrices, \dots) usually depend on the fine triangulation $\triangulation$, the size of which scales with $\Order(1/\eps)$.
Correspondingly all evaluations involving these quantities (scalar products, operator inversions, \dots) become increasingly expensive for small scales $\eps$;
$(ii)$ the generation of the reduced space $Q_\red$ usually requires the discrete problem \eqref{equation::problem::discrete_solution} to be solved for several parameters $\param$.
In the context of multi-scale problems, however, it is usually only feasible to carry out very few global solutions of the original problem on the fine triangulation, if at all.

The first of these shortcomings has been successfully addressed by the \underline{l}ocalized \underline{r}educed \underline{b}asis \underline{m}ulti-\underline{s}cale (LRBMS) method introduced in \cite{KOH2011,AHKO2012}, which takes advantage of the coarse partition and the underlying discretization to form local reduced spaces $Q_\red^\Element \subset Q_h^{k, \Element}$ on each coarse element instead of the usual global RB space.
This allows to balance the computational effort of the off-line versus the on-line phase by selecting an appropriate size of the coarse partition $\Triangulation$ depending on the multi-scale parameter $\eps$ and the real-time or multi-query context of the application.
But it was also shown in \cite{AHKO2012}, that the standard residual based estimator which is usually used in the context of RB methods does not scale well in the context of multi-scale problems.
We thus finalize the definition of the LRBMS method in this contribution with an efficiently off-line/on-line decomposable a-posteriori error estimator (see the next section).
The second shortcoming of classical RB methods is addressed in section \ref{section::online_enrichment}, where we propose an on-line enrichment extension for the LRBMS method based on the presented error estimator, which finally turns the LRBMS method into a fully capable multi-scale method.

\textbf{The reduced problem.} Let us assume that we are given a local reduced space $Q_\red^\Element \subset Q_h^{k, \Element}$ on each coarse element $\Element \in \Triangulation$.
Those are usually, but not necessarily, spanned by solutions of \eqref{equation::problem::discrete_solution} and of low dimension, i.e., $\dim Q_\red^\Element \ll \dim Q_h^{k, \Element}$.
We then define the \emph{coarse reduced space} $Q_\red(\Triangulation) \subset Q_h^k$ by
\begin{align}
  Q_\red(\Triangulation) &:=
      \big\{
        q \in Q_h^k
      \;\big|\;
        \restrict{q}{\Element} \in Q_\red^\Element \quad \forall \Element \in \Triangulation
      \big\}
    \;\;\subset\;\; Q_h^k \komma
  \notag
\end{align}
and obtain the reduced solution by a Galerkin projection of \eqref{equation::problem::discrete_solution} onto $Q_\red(\Triangulation)$.

\begin{definition}[Reduced solution]
  We define the \emph{reduced solution} $p_\red: \Params \to Q_\red(\Triangulation)$ for a parameter $\param \in \Params$, such that
  \begin{align}
    b_h(p_\red(\param), q_\red; \param) = l(q_\red) &&\text{for all } q_\red \in Q_\red(\Triangulation) \punkt
    \label{equation::problem::reduced_solution}
  \end{align}
\end{definition}

\textbf{Off-line/on-line decomposition.} The well-known off-line/on-line decomposition of RB methods relies on an affine parameter dependence of the data functions (see section \ref{section::problem_formulation}) that carries over to the discrete bilinear form, i.e., there exist $\varXi$ nonparametric component bilinear forms $b_{h, \xi}: H^1(\triangulation) \times H^1(\triangulation) \to \R$, for $1 \leq \xi \leq \varXi$, such that $b_h(p, q; \param) = \sum_{\xi = 1}^\varXi \theta_\xi(\param) b_{h,\xi}(p, q)$, with $\theta_\xi$ being the coefficients of the corresponding decomposition of $\lambda$.
In the standard RB framework, given a basis $\big\{ q_{\red, 1}, \dots, q_{\red, N} \big\}$ of $Q_\red$, with $N := \dim Q_\red$, one would precompute dense Gramian matrices $\underline{b_{h,\xi}} \in \R^{N \times N}$, given by $\big(\underline{b_{h, \xi}}\big)_{i, j} := b_{h, \xi}(q_{\red, i}, q_{\red, j})$ during the off-line phase.
In the on-line phase, given a parameter $\param \in \Params$, one would obtain the reduced system matrix by a low dimensional summation $\underline{b_h(\param)} := \sum_{\xi = 1}^\varXi \theta_\xi(\param) \underline{b_{\xi, h}} \in \R^{N \times N}$, which only involves evaluations of the scalar coefficient functions $\theta_\xi$ for the current parameter.
The resulting dense reduced system is of size $N \times N$ and does not depend on the dimension of $Q_h^k$.
It can thus be solved with a complexity of $\Order(N^3)$ in the on-line phase.
Denoting by $N_h := \dim Q_h^k$ and $N_h^\Element := \dim Q_h^k(\triangulation^\Element)$ (with $N_h = \sum_{\Element\in\Triangulation} N_h^\Element$) the number of Degrees of Freedom of the global and the local spaces, respectively, the reduction of $b_h$ during the off-line phase, however, would be of a computational complexity of $\Order(\varXi N^2 N_h)$, which can become too expensive in the context of small scales $\eps$.

The flexible framework of the LRBMS method, however, allows for a more local approach.
Since the affine decomposition of the data function $\lambda$ also carries over to the local bilinear forms $b_h^\Element$ and the coupling bilinear forms $b_h^\Face$, we define local Gramian matrices $\underline{b_{h, \xi}^\Element} \in \R^{N^\Element \times N^\Element}$ on each coarse element $\Element \in \Triangulation$, given by
\begin{align}
  \Big(
      \underline{b_{h, \xi}^\Element}
    \Big)_{i, j} :=
    b_{h, \xi}^\Element(q_{\red, i}^\Element, q_{\red, j}^\Element)
    +\sum_{\Face \in \Faces^\Element}
      b_{h, \xi}^\Face(q_{\red, i}^\Element, q_{\red, j}^\Element)
  \komma\notag
\end{align}
where $\big\{ q_{\red, 1}^\Element, \dots, q_{\red, N^\Element}^\Element \big\}$ denotes a basis of $Q_\red^\Element$ with $N^\Element := \dim Q_\red^\Element$.
In addition, we define coupling Gramian matrices $\underline{b_{h, \xi}^{\Element, \Neighbor}} \in \R^{N^\Element \times N^\Neighbor}$ for all coarse elements $\Element \in \Triangulation$ and their neighbors $\Neighbor \in \Neighbors(\Element)$ by
\begin{align}
  \Big(
      \underline{b_{h, \xi}^{\Element, \Neighbor}}
    \Big)_{i, j} :=
    \sum_{\Face \in \Faces^\Element \cap \Faces^\Neighbor}
      b_{h, \xi}^\Face(q_{\red, i}^\Element, q_{\red, j}^\Neighbor)
  \punkt\notag
\end{align}
In the same manner, we define local vectors $\underline{l^\Element} \in \R^{N^\Element}$, given by $\big(\underline{l^\Element}\big)_j := l^\Element(q_{\red, j}^\Element)$, for all $\Element \in \Triangulation$.
The global Gramian matrices $\underline{b_{h, \xi}} \in \R^{N \times N}$ and the global vector $\underline{l} \in \R^N$, with $N := \sum_{\Element \in \Triangulation} N^\Element$, are then given by a standard DG mapping (with respect to the coarse triangulation) of their local counterparts.
The complexity to directly solve the reduced system in the on-line phase, $\Order\big((\sum_{\Element \in \Triangulation} N^\Element)^3\big)$, is by definition larger than for standard RB methods.
The reduction of $b_h$, however, can be carried out with a complexity of roughly $\Order(\varXi \sum_{\Element \in \Triangulation} {N^\Element}^2 N_h^\Element)$, which can be computed significantly faster than in the standard RB setting described above (since $\varXi |\Triangulation|$ local computations can be carried out in parallel), depending on the choice of $\Triangulation$.

\begin{remark}[Properties of the reduced system]
  The reduced system matrix of the LRBMS method is obviously larger than the one we would obtain for standard RB methods (given that it scales with $|\Triangulation|$).
  On the other hand we obtain a sparse matrix of dense blocks (stemming from the local and coupling blocks and the DG mapping) that is usually much smaller than the system matrix of the high-dimensional problem.
  Thus, if $|\Triangulation|$ is large, the reduced system can also be solved using sparse iterative solvers with a complexity of $\Order\big((\sum_{\Element \in \Triangulation} N_h^\Element)^2\big)$.
\end{remark}

\section{Error analysis}

Our error analysis is a generalization of the ansatz presented in \cite{ESV2010} to provide an estimator for our DG approximation \eqref{equation::problem::discrete_solution} as well as for our LRBMS approximation \eqref{equation::problem::reduced_solution}.
The main idea of the error estimator presented in \cite{Voh2007,ESV2010} is to observe that the approximate DG diffusive flux $-\lambda(\param)\kappa_\eps \gradienth p_h(\param)$ is nonconforming while its exact counterpart belongs to $\Hdiv(\Omega) \subset [L^2(\Omega)]^d$, which denotes the space of vector valued functions the divergence of which lies in $L^2(\Omega)$.
The idea of \cite{Voh2007,ESV2010} is to reconstruct the discrete diffusive flux in a conforming Raviart-Thomas-N\'{e}d\'{e}lec space $V_h^l(\triangulation) \subset\Hdiv(\Omega)$ and compare it to the nonconforming one.
Their error analysis relies on a local conservation property of the reconstructed flux on the fine triangulation $\triangulation$ to prove estimates local to the fine triangulation.

We transfer this concept to the DG discretization defined in section \ref{subsection::discretization} and prove estimates local to the coarse triangulation that are valid for the DG approximation as well as for our LRBMS approximation.
We obtain mild requirements for the coarse triangulation and the local approximation spaces, namely that a local Poincar\'{e} inequality holds on each coarse element and that the constant function $\one$ is present in the local approximation spaces.
The latter is obvious for traditional discretizations and can be easily achieved for the LRBMS approximation.
The estimates are fully off-line/on-line decomposable and can thus be efficiently used for model reduction in the context of the LRBMS.
Preliminary results were published in \cite{OS2014}.

We begin by stating an \emph{abstract energy norm estimate} (see \cite[Lemma 4.1]{ESV2010}) that splits the difference between the weak solution $p \in H^1_0(\Omega)$ of \eqref{equation::problem::weak_solution} and any function $q \in H^1(\triangulation)$ into two contributions.
This abstract estimate does not depend on the discretization and thus leaves the choice of $s$ and $v$ open.
Note that we formulate the following Lemma with different parameters for the energy norm and the weak solution.
The price we have to pay for this flexibility is the additional constants involving $\alphaParam$ and $\gammaParam$, that vanish if $\paramFixed$ and $\param$ coincide.
In the following we denote the product over a space $V(\omega)$, for $\omega \subseteq \Omega$, by $(\cdot, \cdot)_{V, \omega}$ and omit $\omega$ if $\omega = \Omega$; the same holds for the induced norm $\norm{\cdot}{V, \omega}$.
\begin{lemma}[Abstract energy norm estimate]
\label{lemma::error::abstract_energy_norm_estimate}
  Let $p(\param) \in H^1_0(\Omega)$ be the weak solution of \eqref{equation::problem::weak_solution} for $\param \in \Params$ and let $p_h \in H^1(\triangulation)$ and $\paramFixed \in \Params$ be arbitrary.
  Then
  \begin{align}
    \energynorm{p(\param) - p_h}{\paramFixed}
      &\leq \tfrac{1}{\sqrt{\alphaParam}} \Big(
      \inf\limits_{s\in H^1_0(\Omega)} \sqrt{\gammaParam} \energynorm{p_h - s}{\paramFixed}
    \label{equation::error::abstract_energy_norm_estimate}\\
    &+ \inf\limits_{v\in \Hdiv(\Omega)}
      \Big\{
        \sup\limits_{\substack{\varphi\in H^1_0(\Omega)\\\energynorm{\varphi}{\param}=1}}
          \big\{
            \big(
              f - \divergence v, \varphi
            \big)_{L^2}
            -\big(
              \lambda(\param)\kappa_\eps \mydot \gradienth p_h + v, \gradient\varphi
            \big)_{L^2}
          \big\}
      \Big\}\Big)
    \notag\\
    &\leq \tfrac{\sqrt{\gammaParam}}{\sqrt{\alphaParam}} \;\;2\;\; \energynorm{p(\param) - p_h}{\paramFixed}\punkt
    \notag
  \end{align}
\end{lemma}
\begin{proof}
  We mainly follow the proof of \cite[Lemma 4.1]{ESV2010} while accounting for the parameter dependency of the energy norm and the weak solution.
  It holds for arbitrary $\param\in\Params$, $p\in H^1_0(\Omega)$ and $p_h\in H^1(\triangulation)$, that
  \begin{align}
    \energynorm{p - p_h}{\param}
      \leq\inf\limits_{s \in H^1_0(\Omega)}
        \energynorm{p_h - s}{\param}
      +\sup\limits_{\substack{\varphi \in H^1_0(\Omega)\\ \energynorm{\varphi}{\param} = 1}}
        b(p - p_h, \varphi, \param)
    \label{lemma::proof::error::abstract_energy_norm_estimate::1}
  \end{align}
  (see \cite[Lemma 7.1]{Voh2007}) and for the weak solution $p(\param) \in H^1_0(\Omega)$ of \eqref{equation::problem::weak_solution}, that
  \begin{align}
    b(p(\param) - p_h, \varphi; \param)
      &=\big(f, \varphi \big)_{L^2}
      -\big( \lambda(\param) \kappa_\eps \gradienth p_h, \gradient \varphi \big)_{L^2}
      \komma
    \notag\\
    &=\big(f - \divergence v, \varphi \big)_{L^2}
      -\big( \lambda(\param) \kappa_\eps \mydot \gradienth p_h + v, \gradient \varphi \big)_{L^2}
    \label{lemma::proof::error::abstract_energy_norm_estimate::2}
  \end{align}
  for all $\varphi\in H^1_0(\Omega)$ and all $v \in \Hdiv(\Omega)$, where we used the definition of $b$ in the first equality and the fact that $(v, \gradient \varphi)_{L^2} = -(\divergence v, \varphi)_{L^2}$ due to Green's Theorem and $\varphi\in H^1_0(\Omega)$ in the second one.
  Inserting \eqref{lemma::proof::error::abstract_energy_norm_estimate::2} into \eqref{lemma::proof::error::abstract_energy_norm_estimate::1} with $p = p(\param)$ and using the norm equivalence \eqref{equation::problem::parameter_norm_equivalence} then yields the first inequality in \eqref{equation::error::abstract_energy_norm_estimate}.

  To obtain the second inequality we choose $s = p(\param)$ and $v = -\lambda(\param) \kappa_\eps \gradient p(\param)$ in the right hand side of \eqref{equation::error::abstract_energy_norm_estimate} which eliminates the two infimums and leaves us with two terms yet to be estimated arising inside the supremum.
  Using Green's Theorem and the definition of $b$ we observe the vanishing of the first term.
  We estimate the second term as
  \begin{align}
    \big|\big(
          \lambda(\param) \kappa_\eps \gradienth p_h &- \lambda(\param) \kappa_\eps \gradient p(\param), \gradient\varphi
        \big)_{L^2}\big|
    \notag\\
      &=\big|\big(
          (\lambda(\param) \kappa_\eps)^{1/2} \gradienth(p_h - p(\param)), (\lambda(\param) \kappa_\eps)^{1/2} \gradient\varphi
        \big)_{L^2}\big|
    \notag\\
    &\leq
        \norm{(\lambda(\param) \kappa_\eps)^{1/2} \gradienth(p_h - p(\param))}{L^2}
        \norm{(\lambda(\param) \kappa_\eps)^{1/2} \gradient\varphi}{L^2}
    \notag\\
    &=
        \energynorm{p_h - p(\param)}{\param} \energynorm{\varphi}{\param}\komma
    \notag
  \end{align}
  where we used the Cauchy-Schwarz inequality and the definition of the energy norm.
  We finally obtain the second inequality of \eqref{equation::error::abstract_energy_norm_estimate} from the bound above by observing that the supremum vanishes (due to $\energynorm{\varphi}{\param} = 1$) and by using the norm equivalence \eqref{equation::problem::parameter_norm_equivalence} again.
\end{proof}

The following theorem states the main localization result and gives an indication on how to proceed with the choice of $v$:
it allows us to localize the estimate of the above Lemma, if $v$ fulfills a local conservation property.
It is still an abstract estimate in the sense that it does not use any information of the discretization and does not yet fully prescribe $s$ and $v$.

\begin{theorem}[Locally computable abstract energy norm estimate]
\label{theorem::error::locally_computable_abstract_energy_norm_estimate}
  Let $p(\param) \in H^1_0(\Omega)$ be the weak solution of \eqref{equation::problem::weak_solution} for $\param \in \Params$, let $s \in H^1_0(\Omega)$ and $p_h \in H^1(\triangulation)$ be arbitrary, let $v \in \Hdiv(\Omega)$ fulfill the \emph{local conservation property} $(\nabla \cdot v, \mathds{1})_{L^2, \Element} = (f, \mathds{1})_{L^2, \Element}$ and let $C_P^\Element > 0$ denote the constant from the Poincar\'{e} inequality $\norm{\varphi - \Pi_0^\Element\varphi}{L^2,\Element}^2 \leq C_P^\Element h_\Element^2 \norm{\gradient\varphi}{L^2,\Element}^2$ for all $\varphi\in H^1(\Element)$ on all $\Element\in\Triangulation$, where $\Pi_l^\omega$ denotes the $L^2$-orthogonal projection onto $\Pk_l(\omega)$ for $l \in \N$ and $\omega \subseteq \Omega$.
  It then holds for arbitrary $\paramFixed, \paramHat \in \Params$, that
  \begin{align}
    \energynorm{p(\param) - p_h}{\paramFixed}
      &\leq \tilde{\eta}(p_h, s, v; \paramFixed, \paramHat)\komma
    \notag
  \end{align}
  with the abstract \emph{global estimator} $\tilde{\eta}(p_h, s,v; \paramFixed, \paramHat)$ defined as
  \begin{align}
    \tilde{\eta}(p_h, s,v; \paramFixed, \paramHat)
      := \tfrac{1}{\sqrt{\alphaParam}}
        \Bigg[
          &\;\,\sqrt{\gammaParam} \Big(\sum_{\Element\in\Triangulation} \tilde{\eta}_\text{\textnormal{nc}}^\Element(p_h, s; \paramFixed)^2\Big)^{1/2}
          +\Big(\sum_{\Element\in\Triangulation} \tilde{\eta}_\text{\textnormal{r}}^\Element(v)^2\Big)^{1/2}
    \notag\\
      &+\tfrac{1}{\sqrt{\alphaParamHat}} \Big(\sum_{\Element\in\Triangulation} \tilde{\eta}_\text{\textnormal{df}}^\Element(p_h, v; \paramHat)^2\Big)^{1/2}
      \quad\quad\quad\quad\quad\quad\quad\quad\;\;\Bigg]
    \notag
  \end{align}
  and the local \emph{nonconformity estimator} defined as $\tilde{\eta}_\text{\textnormal{nc}}^\Element(p_h, s; \paramFixed) := \energynorm{p_h - s}{\paramFixed, \Element}$, the local \emph{residual estimator} defined as $\tilde{\eta}_\text{\textnormal{r}}^\Element(v) := ({C_P^\Element}/{c_\eps^\Element})^{1/2} h_\Element \norm{f - \divergence v}{L^2,\Element}$ and the local \emph{diffusive flux estimator} defined as $\tilde{\eta}_\text{\textnormal{df}}^\Element(p_h, v; \paramHat) := \norm{(\lambda(\paramHat) \kappa_\eps )^{-1/2}\big( \lambda(\param) \kappa_\eps \gradienth p_h + v \big)}{L^2, \Element}$ for all coarse elements $\Element\in\Triangulation$, where $c_\eps^\Element := \min_{\element\in\triangulation^\Element} c_\eps^\element$.
\end{theorem}
\begin{proof}
  We loosely follow the proof of \cite[Theorem 3.1]{ESV2010} while accounting for the parameter dependency and the coarse triangulation.
  Fixing an arbitrary $s\in H^1_0(\Omega)$ in \eqref{equation::error::abstract_energy_norm_estimate} and localizing with respect to the coarse triangulation yields
  \begin{align}
    \energynorm{p(\param) - p_h}{\paramFixed}
      &\leq \tfrac{1}{\sqrt{\alphaParam}}
      \Big(
      \sqrt{\gammaParam}
      \sqrt{\sum_{\Element\in\Triangulation} \energynorm{p_h - s}{\paramFixed, \Element}^2}
    \label{theorem::proof::error::locally_computable_abstract_energy_norm_estimate::1}\\
    &+\sup\limits_{\substack{\varphi\in H^1_0(\Omega)\\ \energynorm{\varphi}{\param}=1}}
      \Big\{
        \sum_{\Element\in\Triangulation}
          \Big(
            \underbrace{\big(
              f - \divergence v, \varphi
            \big)_{L^2,\Element}}_{:=(i)}
            -\underbrace{\big(
              \lambda(\param) \kappa_\eps \mydot\gradienth p_h + v, \gradient\varphi
            \big)_{L^2,\Element}}_{:=(ii)}
          \Big)
      \Big\}
      \Big)
    \notag
  \end{align}
  which leaves us with two local terms we will estimate separately.
  \begin{enumerate}
    \item[$(i)$] Since $(f - \divergence v, \Pi_0^\Element \varphi)_{L^2, \Element} = 0$ due to the local conservation property of $v$ we can estimate the first term as
      \begin{align}
        \big|\big(
              f - \divergence v, \varphi
            \big)_{L^2,\Element} \big|
          &\leq \norm{f - \divergence v}{L^2, \Element} \norm{\varphi - \Pi_0^\Element \varphi}{L^2, \Element}
        \notag\\
          &\leq \sqrt{C_P^\Element} h_\Element \Big(\max_{\element\in\triangulation^\Element} \frac{1}{c_\eps^\element}\Big)^{-1} \norm{f - \divergence v}{L^2,\Element}
            \energynorm{\varphi}{\param, \Element}\komma
        \notag
      \end{align}
      where we used the Cauchy-Schwarz inequality, the Poincar\'{e} inequality and the local norm equivalence \eqref{equation::problem::parameter_norm_equivalence} on all $\element\in\triangulation^\Element$.
    \item[$(ii)$] We estimate the second term as
    \begin{align}
      \big|\big(
            \lambda(\param) \kappa_\eps \gradienth p_h + v, \gradient \varphi
          \big)_{L^2,\Element}\big|
        \leq &\norm{(\lambda(\param) \kappa_\eps)^{-1/2}\big( \lambda(\param) \kappa_\eps \gradienth p_h + v \big)}{L^2, \Element}
          \energynorm{\varphi}{\param, \Element}
      \notag\\
      \leq \sqrt{\alpha(\param, \paramHat)}^{-1} &\norm{(\lambda(\paramHat) \kappa_\eps)^{-1/2}\big( \lambda(\param) \kappa_\eps \gradienth p_h + v \big)}{L^2, \Element}
          \energynorm{\varphi}{\param, \Element}
      \notag
    \end{align}
    using the Cauchy-Schwarz inequality, the definition of the local energy semi-norms and the parameter equivalency from section \ref{section::problem_formulation}.
  \end{enumerate}
  Inserting the last two inequalities in \eqref{theorem::proof::error::locally_computable_abstract_energy_norm_estimate::1} and using the Cauchy-Schwarz inequality yields
  \begin{align}
    \energynorm{p(\param) - p_h}{\paramFixed}
      &\leq \tfrac{1}{\sqrt{\alphaParam}}
      \Big( \sqrt{\gammaParam} \Big(\sum_{\Element\in\Triangulation} \tilde{\eta}_\text{nc}^\Element(p_h, s; \paramFixed)^2\Big)^{1/2}
    \notag\\
      &+\sup\limits_{\substack{\varphi\in H^1_0(\Omega)\\\energynorm{\varphi}{\param}=1}}
        \Big[
          \underbrace{\sum_{\Element\in\Triangulation}\big(
              \tilde{\eta}_\text{r}^\Element(v)
              + \tfrac{1}{\sqrt{\alphaParamHat}} \; \tilde{\eta}_\text{df}^\Element(p_h, v; \paramHat)
            \big)
            \energynorm{\varphi}{\param, \Element}}_{(iii)}
        \Big]^2\Big)\komma
    \notag
  \end{align}
  using the definition of the local estimators and of $c_\eps^\Element$.
  Using the Cauchy-Schwarz inequality again we can further estimate $(iii)$ as
  \begin{align}
      (iii) \leq \Bigg[
          \Big( \sum_{\Element\in\Triangulation} \tilde{\eta}_\text{r}^\Element(v)^2 \Big)^{1/2}
          + \tfrac{1}{\sqrt{\alphaParamHat}} \Big( \sum_{\Element\in\Triangulation} \tilde{\eta}_\text{df}^\Element(p_h, v; \paramHat)^2 \Big)^{1/2}
        \Bigg] \energynorm{\varphi}{\param}\punkt
    \notag
  \end{align}
  The previous two inequalities combined give the final result, since the supremum vanishes due to $\energynorm{\varphi}{\param} = 1$.
\end{proof}

\begin{remark}[Properties of the locally computable abstract energy norm estimate]
  In contrast to the estimator proposed in \cite{ESV2010} the above estimate is local with respect to $\Triangulation$, not $\triangulation$.
  Choosing $\Triangulation = \triangulation$ we obtain nearly the same estimate as the one in \cite{ESV2010} for the pure diffusion case (apart from a slightly less favorable summation).
  In general, however, we can only expect $\tilde{\eta}_\text{\textnormal{r}}$ to be superconvergent if we refine $\Triangulation$ along with $\triangulation$ (see Section \ref{subsection::experiments_convergence}), thus keeping the ration $H/h$ fixed.
\end{remark}

What is left now in order to turn the abstract estimate of Theorem \ref{theorem::error::locally_computable_abstract_energy_norm_estimate} into a fully computable one is to specify $s$ and $v$, given a DG solution $p_h$.
We will do so in the following paragraphs, finally using the knowledge that $p_h$ was computed using our DG discretization.

\textbf{Oswald interpolation.} The form of the nonconformity indicator in Theorem \ref{theorem::error::locally_computable_abstract_energy_norm_estimate} already indicates how to choose $s$: it should be close to $p_h$, in order to minimize $\tilde{\eta}_\text{nc}$, and it should be computable with reasonable effort.
Both requirements are met by the Oswald interpolation operator, that goes back to \cite{KP2003} (in the context of a-posteriori error estimates; see also \cite[Section 2.5]{ESV2010} and the references therein).
Given any nonconforming approximation $p_h \in Q_h^k(\Triangulation) \not\subset H^1_0(\Omega)$ we choose $s \in H^1_0(\Omega)$ as a conforming reconstruction of $p_h$ by the \emph{Oswald interpolation operator} $I_\text{os} : Q_h^k(\Triangulation) \to Q_h^k(\Triangulation) \cap H^1_0(\Omega)$ which is defined by prescribing its values on the Lagrange nodes $\vertex$ of the triangulation:
we set $I_\text{os}[p_h](\vertex) := p_h^\element(\vertex)$ inside any $\element\in\triangulation$ and
\begin{align}
  I_\text{os}[p_h](\vertex) := \tfrac{1}{|\triangulation^v|} \sum_{\element \in \triangulation^\vertex} p_h^\element(\vertex)
    &&\text{for all inner nodes of } \triangulation \text{ and}&& I_\text{os}[p_h](\vertex) := 0
  \notag
\end{align}
for all boundary nodes of $\triangulation$, where $\triangulation^v \subset \triangulation$ denotes the set of all simplices of the fine triangulation which share $\vertex$ as a node.

We continue with the specification of $v$, which is a bit more involved.
The only formal requirement we have is for $v$ to fulfill the local conservation property on each coarse element, although the diffusive flux estimator already gives a good hint on the specific form of $v$ (namely, to be close to $-\lambda(\param) \kappa_\eps \gradienth p_h(\param)$).
A particular choice is given by the element-wise diffusive flux reconstruction (with respect to the fine triangulation) that was proposed in \cite{ESV2010}, which fulfills the local conservation property on the coarse elements if properly defined with respect to our DG discretization.

\textbf{Diffusive flux reconstruction.} We will reconstruct a conforming diffusive flux approximation $u_h(\param) \in \Hdiv(\Omega)$ of the nonconforming discrete diffusive flux $-\lambda(\param) \kappa_\eps \gradienth p_h(\param) \not\in \Hdiv(\Omega)$ in a conforming discrete subspace $V_h^l(\triangulation) \subset \Hdiv(\Omega)$, namely the \emph{Raviart-Thomas-N\'{e}d\'{e}lec} space of vector functions (see \cite{ESV2010} and the references therein), defined for $k - 1 \leq l \leq k$ by
\begin{align}
  V_h^l(\triangulation)
    :=\big\{
        v \in \Hdiv(\Omega)
      \big|
        \restrict{v}{\element} \in [\Pk_l(\element)]^d + \boldsymbol{x} \Pk_l(\element)
        \quad\forall\element\in\triangulation
      \big\}\punkt
  \notag
\end{align}
See \cite[Section 2.4]{ESV2010} and the references therein  for a detailed discussion of the role of the polynomial degree $l$, the properties of elements of $V_h^l(\triangulation)$ and the origin of the use of diffusive flux reconstructions in the context of error estimation in general.
We define the \emph{diffusive flux reconstruction} operator $R_h^l: \Params \to [Q_h^k(\Triangulation) \to V_h^l(\triangulation)]$, $\param \mapsto \big[ q_h \mapsto R_h^l[q_h; \param]\big]$, by locally specifying $R_h^l[q_h; \param] \in V_h^l(\triangulation)$, such that
\begin{align}
  \big(
        R_h^l[q_h; \param] \mydot n_\face, q
      \big)_{L^2, \face}
      &= b_c^\face(q_h, q; \param) + b_p^\face(q_h, q; \param)
    &&\text{for all } q \in \Pk_l(\face)
  \label{equation::error::flux_reconstruction::1}
\intertext{and all $\face \in \faces^\element$ and}
  \big(
        R_h^l[q_h; \param], \gradient q
      \big)_{L^2, \element}
      &= -b^\element(q_h, q; \param)
        -\vartheta\sum_{\face \in \faces^\element}
          b_c^\face(q, q_h; \param)
    &&\text{for all } \gradient q \in [\Pk_{l-1}(\element)]^d
  \label{equation::error::flux_reconstruction::2}
\end{align}
with $q \in \Pk_l(\element)$ for all $\element \in \triangulation$, where $\vartheta$ is given by the local discretization inside each coarse element and by $\vartheta = 1$ on all fine faces that lie on a coarse face.
The next Lemma shows that this reconstruction of the diffusive flux is sensible for the DG solution as well as the reduced solution, since the reconstructions of both fulfill the requirements of Theorem \ref{theorem::error::locally_computable_abstract_energy_norm_estimate}.

\begin{lemma}[Local conservativity]
  \label{lemma::error::local_conservativity}
  Let $\one \in Q_\red^\Element \subset Q_h^{k, \Element}$ for all $\Element \in \Triangulation$ and let $p_h(\param) \in Q_h^k(\Triangulation)$ be the DG solution of \eqref{equation::problem::discrete_solution} and $p_\red(\param) \in Q_\red(\Triangulation)$ the reduced solution of \eqref{equation::problem::reduced_solution} for a parameter $\param \in \Params$ and let $u_h(\param) := R_h^l[p_h(\param); \param] \in V_h^l(\triangulation)$ and $u_\red(\param) := R_h^l[p_\red(\param); \param] \in V_h^l(\triangulation)$ denote their respective diffusive flux reconstructions.
  It then holds that $u_h(\param)$ and $u_\red(\param)$ fulfill the local conservation property of Theorem \ref{theorem::error::locally_computable_abstract_energy_norm_estimate}, i.e.,
  \begin{align}
    \big( \divergence u_h(\param), \one \big)_{L^2, \Element} = \big( f, \one \big)_{L^2, \Element} = \big( \divergence u_\red(\param), \one \big)_{L^2, \Element}
      &&\text{for all } \Element \in \Triangulation.
      \notag
  \end{align}
\end{lemma}
\begin{proof}
  We follow the ideas of \cite[Lemma 2.1]{ESV2010} while accounting for the coarse triangulation.
  Let $\one^\Element \in Q_h^k(\Triangulation)$ be an indicator for $\Element\in\Triangulation$, such that $\restrict{\one^\Element}{\Element} = \one \in Q_h^{k, \Element}$ and zero everywhere else.
  It then holds with $* = h, \red$, that
  \begin{align}
    \big( \divergence u_*(\param), \one \big)_{L^2, \Element}
      &= \sum_{\element \in \triangulation^\Element}
        \big[
          \big( u_*(\param) \mydot n, \one \big)_{L^2, \boundary\element}
          -\big( u_*(\param), \nabla \one )\big)_{L^2, \element}
        \big]
    \notag\\
      &= b_h(p_*(\param), \one^\Element; \param)
      = \big(f, \one)_{L^2, \Element}
    \notag\komma
  \end{align}
  for all $\Element\in\Triangulation$, where we used Green's Theorem in the first equality, the definition of the diffusive flux reconstruction, \eqref{equation::error::flux_reconstruction::1}, \eqref{equation::error::flux_reconstruction::2}, and the definition of $\one^\Element$ and $b_h$ in the second and the fact, that $\one \in Q_h^{k, \Element}$ and $p_*$ solves \eqref{equation::problem::discrete_solution} or \eqref{equation::problem::reduced_solution}, respectively, in the third.
\end{proof}

Inserting the Oswald interpolation for $s$ and the diffusive flux reconstruction for $v$ in Theorem \ref{theorem::error::locally_computable_abstract_energy_norm_estimate} then yields a locally computable energy estimate for the DG as well as the reduced solution.

\begin{corollary}[Locally computable energy norm estimate]
  \label{corollary::error::locally_computable_energy_norm_estimate}
  Let $p(\param) \in H^1_0(\Omega)$ be the weak solution of \eqref{equation::problem::weak_solution}, let $p_h(\param) \in Q_h^1(\Triangulation)$ be the DG solution of \eqref{equation::problem::discrete_solution}, let $p_\red(\param) \in Q_\red(\Triangulation)$ be the reduced solution of \eqref{equation::problem::reduced_solution} and let $R_h^l$ denote the diffusive flux reconstruction operator.
  Let the assumptions of Theorem \eqref{theorem::error::locally_computable_abstract_energy_norm_estimate} and Lemma \eqref{lemma::error::local_conservativity} be fulfilled and let $\paramFixed, \paramHat \in \Params$ be arbitrary.
  It then holds, that
  \begin{align}
    \energynorm{p(\param) - p_h(\param)}{\paramFixed} &\leq \eta(p_h(\param); \param, \paramFixed, \paramHat)\komma
    \notag\\
    \energynorm{p(\param) - p_\red(\param)}{\paramFixed} &\leq \eta(p_\red(\param); \param, \paramFixed, \paramHat).
    \notag
  \end{align}
  with
  \begin{align}
    \eta(\cdot; \param, \paramFixed, \paramHat)
      := \tfrac{1}{\sqrt{\alphaParam}}
        \Bigg[
          &\;\,\sqrt{\gammaParam} \Big(\sum_{\Element\in\Triangulation} \eta_\text{\textnormal{nc}}^\Element(\cdot; \paramFixed)^2\Big)^{1/2}
          +\Big(\sum_{\Element\in\Triangulation} \eta_\text{\textnormal{r}}^\Element(\cdot; \param)^2\Big)^{1/2}
    \notag\\
      &+\tfrac{1}{\sqrt{\alphaParamHat}} \Big(\sum_{\Element\in\Triangulation} \eta_\text{\textnormal{df}}^\Element(\cdot; \param, \paramHat)^2\Big)^{1/2}
      \quad\quad\quad\quad\quad\quad\quad\quad\;\Bigg]
    \notag
  \end{align}
  and
  \begin{align}
    \eta_\text{\textnormal{nc}}^\Element(\cdot; \paramFixed) := \tilde{\eta}_\text{\textnormal{nc}}^\Element(\cdot, I_\text{os}[\cdot]; \paramFixed)\komma
      &&\eta_\text{\textnormal{r}}^\Element(\cdot; \param) := \tilde{\eta}_\text{\textnormal{r}}^\Element(R_h^l[\cdot; \param])\komma
      &&\eta_\text{\textnormal{df}}^\Element(\cdot; \param, \paramHat) := \tilde{\eta}_\text{\textnormal{df}}^\Element(\cdot, R_h^l[\cdot; \param]; \paramHat)\punkt
    \notag
  \end{align}
\end{corollary}

\textbf{Local efficiency.} The global efficiency of the abstract estimate was already shown in Lemma \ref{lemma::error::abstract_energy_norm_estimate} (again note, that $\gammaParam = \alphaParam = 1$ if $\param$ and $\paramFixed$ coincide); see \cite[Remarks 4.2 and 4.3]{ESV2007} for a discussion of the global efficiency of the abstract estimate.
We also state a local efficiency of the local estimates using the knowledge of the discretization, the Oswald interpolation and the diffusive flux reconstruction.
To do so we further localize our estimates with respect to the fine triangulation and apply the ideas of \cite{ES2008,ESV2007,ESV2010}.
We denote by $\lesssim$ a proportionality relation between to quantities $a$ and $b$ in the sense that $a \lesssim b :\Longleftrightarrow a \leq Cb$, where the positive constant $C$ depends on the space dimension, the polynomial degree $k$, the polynomial degree of $f$, the shape-regularity of $\triangulation$ and the DG parameters $\sigma$ and $\vartheta$, only.
We additionally denote the set of all fine elements that touch $\Element \in \Triangulation$ by $\tildetriangulation^\Element := \{ \element \in \triangulation \,|\, \element \cap \Element \neq \emptyset \}$, the set of all fine faces that touch $\Element$ by $\tildefaces^\Element := \{ \face \in \faces \,|\, \exists \element \in \triangulation^\Element : \face \cap \element \neq \emptyset \}$ and the weighted jump seminorms $\jumpnorm{\cdot}{\cdot, \mathcal{F}} : \Params \to [H^1(\triangulation) \to \R]$, $\param \mapsto [q \mapsto \jumpnorm{q}{\param, \mathcal{F}}]$ and $\jumpnorm{\cdot}{p, \cdot, \mathcal{F}} : \Params \to [H^1(\triangulation) \to \R]$, $\param \mapsto [q \mapsto \jumpnorm{q}{p, \param, \mathcal{F}}]$ for any subset $\mathcal{F} \subset \faces$, all $\param \in \Params$ and $q \in H^1(\triangulation)$ by
\begin{align}
  \jumpnorm{q}{\param, \mathcal{F}} := \Big( \sum_{\face \in \mathcal{F}} \norm{\jump{\big( \lambda(\param) \kappa_\eps \gradient q \big) \mydot n_\face}}{L^2, \face} \Big)^{1/2}
    \;\;\text{and}\quad \jumpnorm{q}{p, \param, \mathcal{F}} := \Big( \sum_{\face \in \mathcal{F}} b_p^\face(q, q; \param) \Big)^{1/2} \komma
  \notag
\end{align}
respectively.
We analogously denote the set of all fine elements that touch $\element \in \triangulation$  by $\tildetriangulation^\element$ and the set of all fine faces that touch $\element$ by $\tildefaces^\element$ and define
\begin{align}
  \tilde{c}_\eps^\Element &:= \min_{\element \in \tildetriangulation^\Element} c_\eps^\element \komma
    &&& \overline{h_\Element} &:= \max_{\element \in \triangulation^\Element} h_t \komma
    &&& \overline{\omega}^\Element &:= \big( \max_{\face \in \faces^\Element} {\omega_\face^+}^2 \big)^{1/2} \komma
  \notag\\
  C_\eps^\Element &:= \max_{\element \in \triangulation} C_\eps^\element \komma
    &&& \underline{h_\Element} &:= \min_{\element \in \triangulation^\Element} h_t \komma
    &&& \overline{C_p^\Element} &:= \max_{\element \in \triangulation} C_p^\element \komma
  \notag\\
  \overline{\underline{C}}_\eps^\Element &:=
    \max_{\element \in \triangulation^\Element}
      \big(
        (\max_{\neighbor \in \element \cup \Neighbors(t)} \tfrac{C_\eps^\neighbor}{c_\eps^\neighbor})^2
      \big)
  \hnS\hnS\hnS\hnS\hnS\hnS\hnS\hnS\hnS\hnS\hnS\hnS\hnS\hnS\hnS\hnS\hnS\hnS\hnS\hnS
  \notag
\end{align}
for all $\Element \in \Triangulation$, where $C_p^\element > 0$ denotes the Poincar\'{e} constant on $\element \in \triangulation$, defined analogously to $C_p^\Element$.

\begin{theorem}[Local efficiency of the locally computable energy norm estimate]
  With the notation and assumptions from Corollary \ref{corollary::error::locally_computable_energy_norm_estimate}, let $f$ be polynomial and $\max_{\element \in \triangulation} h_\element \leq 1$.
  It then holds with $* = h, \red$, respectively, that
  \begin{align}
    \eta_\text{\textnormal{nc}}^\Element(p_*(\param); \paramFixed)
      &\lesssim \big({C_\eps^\Element}/{\tilde{c}_\eps^\Element}\big)^{1/2} \;\, \jumpnorm{p(\param) - p_*(\param)}{p, \paramFixed, \tildefaces^\Element}
    \komma\notag\\
    \eta_\text{\textnormal{r}}^\Element(p_*(\param); \param)
      &\lesssim \sqrt{\gammaParam} ({C_p^\Element} / c_\eps^\Element)^{1/2} h_\Element \Big[ \quad\quad\quad\quad\quad
        C_\eps^\Element \underline{h_\Element}^{-1} \;\, \energynorm{p(\param) - p_*(\param)}{\paramFixed, \Element}
    \notag\\
      &\quad\quad\quad\quad\quad\quad\quad\quad+
        \overline{C_p^\Element}^{1/2} {c_\eps^\Element}^{-1/2}
          \Big(\quad\;\;\,
            \overline{\omega}^\Element \; \overline{h_\Element}^{1/2} \, \jumpnorm{p_*(\param)}{\paramFixed, \faces^\Element}
    \notag\\
      &\quad\quad\quad\quad\quad\quad\quad\quad\quad\quad\quad\quad\quad\quad\quad
            +{C_\eps^\Element}^{1/2}\, \sigma^{1/2} \, \jumpnorm{p(\param) - p_*(\param)}{p, \paramFixed, \faces^\Element}
          \Big) \Big]
    \notag\\
    \eta_\text{\textnormal{df}}^\Element(p_*(\param); \param, \paramHat)
      &\lesssim
      \sqrt{\gammaParamHat}\sqrt{\gammaParam}\; {\overline{\underline{C}}_\eps^\Element}^{1/2}
        \Big(\;\;\;
          \energynorm{p(\param) - p_*(\param)}{\paramFixed, \Element}
    \notag\\
      &\quad\quad\quad\quad\quad\quad\quad\quad\quad\quad\quad\quad\;\;\,
          + \jumpnorm{p(\param) - p_*(\param)}{p, \paramFixed, \faces^\Element}
        \Big)
    \notag
  \end{align}
  for all coarse elements $\Element \in \Triangulation$.
\end{theorem}
\begin{proof}
  We estimate each local estimator separately.
  It holds for the nonconformity estimator, that
  \begin{align}
    \eta_\text{\textnormal{nc}}^\Element(p_*(\param); \paramFixed)
      &= \Big(
        \sum_{\element \in \triangulation^\Element}
          \energynorm{p_*(\param) - I_\text{\textnormal{os}}[p_*(\param)]}{\paramFixed, \element}^2
      \Big)^{1/2}
    \notag\\
      &\lesssim \Big(
        \sum_{\element \in \triangulation^\Element}
          C_\eps^\element \min_{\element \in \tildetriangulation^\element}(c_\eps^\element)^{-1} \jumpnorm{p(\param) - p_*(\param)}{p, \paramFixed, \tildefaces^\element}^2
      \Big)^{1/2}
    \notag\\
      &\leq \big({C_\eps^\Element}/{\tilde{c}_\eps^\Element}\big)^{1/2} \jumpnorm{p(\param) - p_*(\param)}{p, \paramFixed, \tildefaces^\Element}
    \komma\notag
  \end{align}
  where we use the definition of $\eta_\text{\textnormal{nc}}^\Element[p_*(\param)]$ and $\triangulation^\Element$ in the equality and the arguments of \cite[Proof of Theorem 3.2]{ESV2010} in the first and the definition of $C_\eps^\Element$, $\tilde{c}_\eps^\Element$ and $\tildefaces^\Element$ in the second inequality.

  It holds for the residual estimator that
  \begin{align}
    \eta_\text{\textnormal{r}}^\Element(p_*(\param); \param)
      \leq \big({C_p^\Element}/{c_\eps^\Element}\big)^{1/2} h_\Element
        \Big(
          &\underbrace{\norm{f - \divergence \big( \lambda(\param) \kappa_\eps \gradient p_*(\param) \big)}{L^2, \Element}}_{:= (i)}
    \notag\\
          +&\underbrace{\norm{\divergence \big( \lambda(\param) \kappa_\eps \gradient p_*(\param) + u_*(\param) \big)}{L^2, \Element}}_{:= (ii)}
        \Big)
    \komma\notag
  \end{align}
  where we used the definition of $\eta_\text{\textnormal{r}}^\Element[u_*(\param)]$ and the triangle inequality, which leaves us with two terms we will estimate separately.
  \begin{itemize}
    \item[$(i)$] The first term can be estimated as follows, using the definition of $\triangulation^\Element$ and the arguments of \cite[Proposition 3.3]{ES2008} in the first and the definition of $C_\eps^\Element$ and $\triangulation^\Element$ and the fact that $\max_{\element \in \triangulation} h_\element \leq 1$ in the second inequality:
      \begin{align}
        (i) \lesssim \Big(
              \sum_{\element \in \triangulation^\Element}
                C_\eps^\element h_t^{-2} \energynorm{p(\param) - p_*(\param)}{\param, \element}^2
            \Big)^{-1/2}
          \leq C_\eps^\Element \underline{h_\Element}^{-1} \energynorm{p(\param) - p_*(\param)}{\param, \Element}
        \notag
      \end{align}
    \item[$(ii)$] The second term can be estimated as
      \begin{align}
        (ii) &\lesssim \Big[
            \;\;\,\sum_{\element \in \triangulation}
              C_p^\element h_\element {c_\eps^\element}^{-1}
              \sum_{\face \in \faces^\element}
                {\omega_\face^+}^2 \norm{\jump{\big( \lambda(\param) \kappa_\eps \gradient p_*(\param) \big) \mydot n_\face}}{L^2, \face}^2
        \notag\\
            &\quad\;\;+\sum_{\element \in \triangulation}
              C_p^\element \sigma C_\face^\element {c_\face^\element}^{-1}
              \jumpnorm{p(\param) - p_*(\param)}{p, \param, \faces^\element}^2
          \;\;\Big]^{1/2}
        \notag\\
        &\lesssim\; \quad\;\;\overline{C_p^\Element}^{1/2} {c_\eps^\Element}^{-1/2}\; \overline{\omega}^\Element \; \overline{h_\Element}^{1/2}
            \jumpnorm{p_*(\param)}{\param, \faces^\Element}
        \notag\\
        &\quad\;\;+\;\overline{C_p^\Element}^{1/2} {c_\eps^\Element}^{-1/2}\; {C_\eps^\Element}^{1/2}\; \sigma^{1/2}
          \jumpnorm{p(\param) - p_*(\param)}{p, \param, \faces^\Element}
        \komma\notag
      \end{align}
      using the definition of $\triangulation^\Element$ and the arguments of \cite[Proof of Theorem 3.2]{ESV2010} in the first and the definition of $\overline{C_p^\Element}$, $\overline{h_\Element}$, $C_\eps^\Element$, $c_\eps^\Element$ and $\triangulation^\Element$ in the second inequality.
  \end{itemize}%
  Applying the norm equivalence \eqref{equation::problem::parameter_norm_equivalence} yields the desired result for the residual estimator.

  Finally, it holds for the diffusive flux estimator, that
  \begin{align}
    \eta_\text{\textnormal{df}}^\Element(p_*(\param); \param, \paramFixed)
      &\leq \sqrt{\gammaParamHat}
        \Big(
          \sum_{\element \in \triangulation^\Element}
            \norm{\big( \lambda(\param) \kappa_\eps \big)^{-1/2} \big( \lambda(\param) \kappa_\eps \gradient p_*(\param) + u_*(\param) \big)}{L^2, \element}^2
        \Big)^{1/2}
    \notag\\
      &\lesssim \sqrt{\gammaParamHat}
        \Big[
          \sum_{\element \in \triangulation^\Element}
            \big(
              \max_{\neighbor \in \element \cup \Neighbors(t)} \tfrac{C_\eps^\neighbor}{c_\eps^\neighbor}
            \big)^2
            \Big(
              \energynorm{p(\param) - p_*(\param)}{\param, \element}
    \notag\\
      &\quad\quad\quad\quad\quad\quad\quad\quad\quad\quad\quad\quad\quad\;\;\;
              +\jumpnorm{p(\param) - p_*(\param)}{p, \param, \faces^\element}
            \Big)^2
        \Big]^{1/2}
    \notag\\
      &\lesssim \sqrt{\gammaParamHat}\; {\overline{\underline{C}}_\eps^\Element}^{1/2}
        \Big(
          \energynorm{p(\param) - p_*(\param)}{\param, \Element}
          + \jumpnorm{p(\param) - p_*(\param)}{p, \param, \faces^\Element}
        \Big)
    \komma\notag
  \end{align}
  where we used the definition of $\eta_\text{\textnormal{df}}^\Element[p_*(\param)]$ and $\triangulation^\Element$ and the parameter equivalence from \eqref{equation::problem::parameter_norm_equivalence} in the first, \cite[Lemma 4.12]{ESV2007} in the second and the definition of $\overline{\underline{C}}_\eps^\Element$ and $\faces^\Element$ in the third inequality.
  Applying the norm equivalence \eqref{equation::problem::parameter_norm_equivalence} again finally yields the desired result for the diffusive flux estimator.
\end{proof}

\section{On-line enrichment}
\label{section::online_enrichment}

As mentioned in Section \ref{section::model_reduction} there are two drawbacks of classical RB methods in the context of parametric multi-scale problems, stemming from the fact that $\dim Q_h^k$ roughly scales with $\Order(\eps^{-1})$: $(i)$ expensive high-dimensional evaluations of global quantities during reduction and orthogonalization and $(ii)$ expensive high-dimensional inversions of \eqref{equation::problem::discrete_solution} during the basis generation.
The first shortcoming was addressed by the LRBMS method introduced in \cite{KOH2011,AHKO2012} and finalized by the local error estimator introduced in the previous section, that can be off-line/on-line decomposed with a computational complexity depending only linearly on $|\triangulation|$ (not shown) while the original estimator presented in \cite{AHKO2012} required the computation of $|\Triangulation|$ global Riesz-representatives for each snapshot.
But also the LRBMS method suffers from the second shortcoming, namely that the basis generation (for instance using an adaptive Greedy procedure) requires the global high-dimensional problem \eqref{equation::problem::discrete_solution} to be solved for a possibly large number of parameters out of a finite training set $\Params_\text{train} \subset \Params$.
In the context of parametric multi-scale problems, however, solving \eqref{equation::problem::discrete_solution} may be prohibitively expensive and one may only have the resources to do so for a very limited amount of parameters, if at all.
Such an RB space constructed out of only very few solution snapshots is usually insufficient for nearly all parameters $\param \in \Params$ that were not included in the basis generation.

\begin{algorithm}
  \footnotesize%
  \caption{Discrete weak Greedy algorithm used in the LRBMS method}
  \label{algorithm::greedy}
  \begin{algorithmic}
    \Require \texttt{ONB}, $k_H \in \N$, $\Params_\text{train} \subset \Params$, $\Delta_\text{greedy} > 0$, $N_\text{greedy} \in \N$
    \Ensure  A local reduced basis $\varPhi^\Element$ for each coarse element $\Element \in \Triangulation$.
    \State \emph{Initialize the local reduced bases with the coarse DG basis:}
    \ForAll{$\Element \in \Triangulation$}
      \State ${\varPhi^\Element}^{(0)} \;\;\, \gets \texttt{ONB}\big(\{\text{DG shape functions of order up to $k_H$ w.r.t. $\Element$}\}\big)$
    \EndFor
    \State ${Q_\red(\Triangulation)}^{(0)} \gets \bigoplus_{\Element \in \Triangulation} \text{span} \big( {\varPhi^\Element}^{(0)} \big)$
    \State $n \gets 0$
    \While{%
      $\max\limits_{\param \in \Params_\text{train}} \; \eta(p_\red(\param); \param, \paramFixed, \paramHat) > \Delta_\text{greedy}$, with $p_\red(\param) \in {Q_\red(\Triangulation)}^{(n)}$ solving \eqref{equation::problem::reduced_solution}\\
      \hspace{7ex}\textbf{and} $n < N_\text{greedy}$
    }
      \State \emph{Compute all reduced quantities w.r.t} ${Q_\red(\Triangulation)}^{(n)}$.
      \State \emph{Find the worst approximated parameter, with $p_\red(\param) \in {Q_\red(\Triangulation)}^{(n)}$ solving \eqref{equation::problem::reduced_solution}}:
      \State $\quad\quad \param_\text{max} \gets \mathop{\arg\max}\limits_{\param \in \Params_\text{train}} \; \eta(p_\red(\param); \param, \paramFixed, \paramHat)$
      \State \emph{Extend the local reduced bases, with $p_h(\param) \in Q_h^k$ solving \eqref{equation::problem::discrete_solution}}:
      \ForAll{$\Element \in \Triangulation$}
        \State ${\varPhi^\Element}^{(n+1)} \gets \texttt{ONB}\big(\{{\varPhi^\Element}^{(n)}, \restrictInline{p_h(\param)}{\Element}\} \big)$
      \EndFor
    \State ${Q_\red(\Triangulation)}^{(n+1)} \gets \bigoplus_{\Element \in \Triangulation} \text{span} \big( {\varPhi^\Element}^{(n+1)} \big)$
    \State $n \gets n + 1$
    \EndWhile\\
    \Return $\big\{ {\varPhi^\Element}^{(n)} \big\}_{\Element \in \Triangulation}$
  \end{algorithmic}
\end{algorithm}

To address this shortcoming we relax the notion of a strict off-line/on-line splitting of the computation in the classical sense.
While the computational complexity of the on-line phase must not depend on any global high-dimensional quantities (namely $\dim Q_h^k$) for RB methods, we allow for high-dimensional but local computations (of order $\dim Q_h^{k,\Element}$, with $\Element \in \Triangulation$) in the context of the LRBMS method.
The idea is as follows: during the off-line phase we initialize the local reduced bases $Q_\red^\Element$ with a DG basis of order up to $k_H \in \N$ with respect to the coarse elements $\Element \in \Triangulation$, thus ensuring that any reduced solution is at least as good as a DG solution on the coarse triangulation.
We then carry out a discrete weak Greedy algorithm based on the error estimator defined in Corollary \ref{corollary::error::locally_computable_energy_norm_estimate}  while allowing only for a limited amount of global solution snapshots, $N_\text{max} \in \N$, and extend the local bases with these snapshots using an orthonormalization algorithm \texttt{ONB} locally on each $\Element \in \Triangulation$.
This procedure is summarized in Algorithm \ref{algorithm::greedy}.

During the on-line phase, given any $\param \in \Params$, we compute a reduced solution $p_\red(\param) \in Q_\red(\Triangulation)$ and efficiently assess its quality using the error estimator.
If the estimated error is above a prescribed tolerance, $\Delta_\text{online} > 0$, we start an intermediate local enrichment phase to enrich the reduced bases in the SEMR (\underline{s}olve $\to$ \underline{e}stimate $\to$ \underline{m}ark $\to$ \underline{r}efine) spirit of adaptive mesh refinement (the procedure is summarized in Algorithm \ref{algorithm::online_enrichment}):
we first compute local error indicators $\eta^\Element(p_\red(\param); \param, \paramFixed, \paramHat)$ for all $\Element \in \Triangulation$, such that $\eta(\cdot; \param, \paramFixed, \paramHat)^2 \leq \sum_{\Element \in \Triangulation} \eta^\Element(\cdot; \param, \paramFixed, \paramHat)^2$, defined as
\begin{align}
    \eta^\Element(\cdot; \param, \paramFixed, \paramHat)^2
      := \tfrac{3}{\sqrt{\alphaParam}}
        \Big[
          \sqrt{\gammaParam} \, \eta_\text{\textnormal{nc}}^\Element(\cdot; \paramFixed)^2
          + \eta_\text{\textnormal{r}}^\Element(\cdot; \param)^2
          + \tfrac{1}{\sqrt{\alphaParamHat}} \, \eta_\text{\textnormal{df}}^\Element(\cdot; \param, \paramHat)^2
        \Big]
    \label{equation::online::local_indicators}
\end{align}
and mark coarse elements $\tildeTriangulation \subseteq \Triangulation$ for enrichment, given a marking strategy \texttt{MARK}.
For each marked $\Element \in \tildeTriangulation$ we solve
\begin{align}
  b_h^{\Element_\delta}(p_h^{\Element_\delta}(\param), q_h; \param) = l_h^{\Element_\delta}(q_h)
  &&\text{for all } q_h \in Q_h^k(\triangulation^{\Element_\delta})
  \label{equation::online::oversampled_problem}
\end{align}
on an overlapping domain $\Element_\delta \supset \Element$ with the insufficient reduced solution $p_\red(\param)$ as dirichlet boundary values on $\boundary \Element_\delta$ to obtain an updated detailed solution $p_h^{\Element_\delta}(\param) \in Q_h^k(\triangulation^{\Element_\delta})$.
Here $Q_h^k(\triangulation^{\Element_\delta})$, $b_h^{\Element_\delta}$ and $l_h^{\Element_\delta}$ are extensions of $Q_h^k(\triangulation^\Element)$, $b_h^\Element$ and $l^\Element$, respectively, to the oversampled domain $\Element_\delta$ while additionally encoding $p_\red(\param)$ as dirichlet boundary values.
We then extend each marked local reduced basis by performing an orthonormalization procedure on $\restrictInline{p_h^{\Element_\delta}(\param)}{\Element}$ with respect to the existing local basis and update all reduced quantities with respect to the newly added basis vector.
We finally compute an updated reduced solution in the updated coarse reduced space and estimate the error again.
We repeat this procedure until the estimated error falls below the prescribed tolerance $\Delta_\text{online}$ or until the prescribed maximum number of iterations, $N_\text{online} \in \N$, is reached.
Possible choices for \texttt{ONB} and \texttt{MARK} are given in Section \ref{subsection::experiments_enrichment}.
Note that first steps in the direction of on-line enrichment were published in \cite{AO2013}.

\begin{algorithm}
  \footnotesize%
  \caption{Adaptive basis enrichment in the intermediate local enrichment phase}
  \label{algorithm::online_enrichment}
  \begin{algorithmic}
    \Require \texttt{MARK}, \texttt{ONB}, $\big\{ \varPhi^\Element \big\}_{\Element \in \Triangulation}$, $p_\red(\param)$, $\param$, $\Delta_\text{online} > 0$, $N_\text{online} \in \N$
    \Ensure  Updated reduced solution and local reduced bases.
    \State ${\varPhi^\Element}^{(0)} \gets \varPhi^\Element$, $\forall \Element \in \Triangulation$
    \State $n \gets 0$
    \While{%
      $\eta(p_\red(\param); \param, \paramFixed, \paramHat) > \Delta_\text{online}$ \textbf{and} $n < N_\text{online}$
    }
      \ForAll{$\Element \in \Triangulation$}
        \State \emph{Compute local error indicator $\eta^\Element(p_\red(\param); \param, \paramFixed, \paramHat)$ according to \eqref{equation::online::local_indicators}.}
      \EndFor
      \State $\tildeTriangulation \gets \texttt{MARK}\big( \Triangulation \big)$
      \ForAll{$\Element \in \tildeTriangulation$}
        \State \emph{Solve \eqref{equation::online::oversampled_problem} for $p_h^{\Element_\delta}(\param) \in Q_h^k(\triangulation^{\Element_\delta})$.}
        \State ${\varPhi^\Element}^{(n+1)} \gets \texttt{ONB}\big(\{{\varPhi^\Element}^{(n)}, \restrictInline{p_h^{\Element_\delta}(\param)}{\Element}\} \big)$
      \EndFor
      \State ${Q_\red(\Triangulation)}^{(n+1)} \gets \bigoplus_{\Element \in \tildeTriangulation} \text{span} \big( {\varPhi^\Element}^{(n+1)} \big) \oplus \bigoplus_{\Element \in \Triangulation \backslash \tildeTriangulation} \text{span} \big( {\varPhi^\Element}^{(n)} \big)$
      \State \emph{Update all reduced quantities w.r.t ${Q_\red(\Triangulation)}^{(n+1)}$.}
      \State \emph{Solve \eqref{equation::problem::reduced_solution} for $p_\red(\param) \in {Q_\red(\Triangulation)}^{(n+1)}$.}
      \State $n \gets n + 1$
    \EndWhile\\
    \Return $p_\red(\param)$, $\big\{ {\varPhi^\Element}^{(n)} \big\}_{\Element \in \Triangulation}$
  \end{algorithmic}
\end{algorithm}

\begin{remark}
  It is also possible to use other methods to compute (or approximate) $p_h(\param)$ during the Greedy procedure, in particular other domain decomposition or multi-scale methods.
  If the resulting approximation does not fulfill the local conservation property of Lemma \ref{lemma::error::local_conservativity}, however, the estimator would have to be replaced during the Greedy algorithm, for instance in the spirit of \cite{PVWW2013}.
  But during the on-line phase the estimator would be valid for any reduced solution, as long as the basis is initialized with at least a constant function.
  In particular one could use variants of the MsFEM (see \cite{EH2009}), the HMM (see \cite{AEEV2012}) or in particular the DG-HMM (see \cite{Abd2012,AH2014}) during the Greedy procedure to generate a coarse reduced basis with approximation properties of order $H$.
  Fine scale features of the solution would then be adaptively recovered during the on-line enrichment phase, if and where needed.
\end{remark}

\begin{remark}
  The computation of the local error indicators in Algorithm \ref{algorithm::online_enrichment} can be efficiently off-line/on-line decomposed (not shown here).
  Once a set of subdomains has been marked, the enrichment can be carried out in parallel without any communication.
  For the update of the reduced quantities only local information and the information on one layer of neighboring fine grid cells is needed.
  The on-line phase, however, requires information on $\Triangulation$ (in particular the number of coarse elements and neighboring information), in contrast to traditional RB methods.
\end{remark}

\begin{remark}
  Our choice of the Greedy algorithm and the adaptive on-line enrichment covers a wide range of scenarios.
  Disabling the on-line enrichment (by setting $N_\text{online} = 0$) and choosing any suitable $\Delta_\text{greedy}$ and $N_\text{greedy}$ yields the standard Greedy basis generation, well known in the RB context.
  Setting $N_\text{greedy} = 0$ and $k_H = 1$, on the other hand, disables the Greedy procedure and merely initializes the reduced bases with the coarse DG basis of order one.
  This is of particular interest in situations where the computation of solutions of the detailed problem during the Greedy procedure might be too costly.
  In that setup nearly all work is done in the adaptive on-line enrichment phase.
  Many other variants are possible, e.g. other local boundary conditions, several marking strategies \texttt{MARK} or orthonormalization algorithms \texttt{ONB} or other stopping criteria; one could also limit the number of intermediate snapshots added to the local bases.
  Depending on these choices the resulting method is then close to existing DD methods (i.e., a DD method with overlapping subdomains, see \cite{QV1999}) or multi-scale methods (i.e., the adaptive iterative multi-scale finite volume method \cite{HJ2011}).
\end{remark}

The LRBMS method with the proposed adaptive on-line enrichment strategy is now suitable for a far wider range of circumstances than standard RB methods or the previously published variant of the LRBMS method.
As mentioned before it can now be applied if the computational power available for the off-line phase is limited by time- or resource constraints.
It can also be applied if the set of training parameters given to the Greedy Algorithm was insufficiently chosen or even if on-line a solution to a parameter is requested that is outside of the original bounds of the parameter space.
In general, the on-line adaptive LRBMS method can be applied whenever the basis that was generated during the off-line phase turns out to not be sufficient for what is required during the on-line phase.

\section{Numerical experiments}

In this section we investigate the performance of the error estimator in the context of the DG discretization defined in Section \ref{subsection::discretization} as well as in the context of the the LRBMS method defined in Section \ref{section::model_reduction}.
We also investigate the performance of the on-line enrichment procedure we proposed in the previous section.

We used several software packages for the implementation: everything concerning the discretization was implemented within the high performance \Cpp{} software framework \texttt{DUNE} \cite{BBD+2008,BBD+2008a} while everything related to model reduction was implemented based on the \texttt{pyMOR} package \cite{pymor}.
Data functions, container and linear solvers were implemented within \dune{stuff} \cite{dunestuff}, discrete function spaces (based on the discretization modules \dune{fem} \cite{DKNO2010} and \dune{pdelab} \cite{dunepdelab}), operators and products for the discretization as well as the error estimator were implemented within \dune{gdt} \cite{dunegdt}.
The coarse triangulation was implemented in \dune{grid-multiscale} \cite{dunegridmultiscale} while the Python bindings forming the bridge between \texttt{DUNE} and \texttt{pyMOR} were implemented in \dune{pymor} \cite{dunepymor}.
Finally a high-level solver for linear elliptic (and possibly parametric) PDEs was implemented in \dune{hdd} \cite{dunehdd} and exposed to \texttt{pyMOR}, where everything related to model reduction was implemented.

The fine triangulations $\triangulation$ used are conforming refinements of triangular grids, represented by instances of \texttt{ALUGrid< 2, 2, simplex, conforming >} (see \cite{DKN2014}).
In the following experiments we choose a DG space $Q_h^{1,\Element}(\triangulation^\Element)$ locally on all coarse elements $\Element \in \Triangulation$.
The resulting discretization thus coincides with the one defined in \cite{ESZ2009,ESV2010} (though in general other choices are possible).
All coarse triangulations $\Triangulation$ used consist of squared elements (though arbitrary shapes are possible).

Most figures in this publication were created or arranged using {T}i{\it k}{Z} \cite{Tan2013,Tan} and pgfplots \cite{Feu}, the colorblind safe colors in Figures \ref{figure::experiments::enrchment::academic::error_evolution} and \ref{figure::experiments::enrichment::spe10::error_evolution} were selected with ColorBrewer \cite{colorbrewer}.

\subsection{Study of the a-posteriori error estimator}
\label{subsection::experiments_convergence}

To study the convergence properties of our estimator we consider two experiments.
The first one serves as an academic example and as a comparison to the work of \cite{ESV2007,ESV2010}.
The second experiment demonstrates the efficiency of the estimator in realistic circumstances.
In both experiments we compute estimator components ${\eta_*}^2 := \sum_{\Element \in \Triangulation} {\eta_*^\Element}^2$, for $* = $, nc, r, df and the estimator $\eta$ as defined in Corollary \ref{corollary::error::locally_computable_energy_norm_estimate}, using a $0th$ order diffusive flux reconstruction ($l = 0$ in \eqref{equation::error::flux_reconstruction::1}, \eqref{equation::error::flux_reconstruction::2}).
If no analytical solution $p(\param)$ is available we approximate the discretization error $\energynorm{p(\param) - p(\param)}{\paramFixed}$ by substituting $p(\param)$ for a discrete solution on a finer grid and by computing all integrals using a high order quadrature on the finer grid.
We denote the efficiency of the estimator, $\eta(p_h(\param); \param, \paramFixed, \paramHat) / \energynorm{p(\param) - p_h(\param)}{\paramFixed} \geq 1$, by ``eff.'' and the average (over all refinement steps) experimental order of convergence of a quantity by ``order''.

\textbf{Academic example.} We consider \eqref{equation::problem::global_pressure} on $\Omega = [-1, 1]^2$ with a parameter space $\Params = [0.1, 1]$, $\kappa_\eps \equiv \text{id}$, $f(x,y) = \tfrac{1}{2} \pi^2 \cos(\tfrac{1}{2} \pi x) \cos(\tfrac{1}{2} \pi y)$, $\lambda(x,y; \param) = 1 + (1 - \param) \cos(\tfrac{1}{2} \pi x) \cos(\tfrac{1}{2} \pi y)$ and homogeneous Dirichlet boundary values.
We study the components of the estimator as well as its efficiency in several circumstances, i.e., for different parameters $\param$, $\paramFixed$, $\paramHat \in \Params$ and triangulations $\triangulation$ and $\Triangulation$.
\begin{table}[b]
  \centering%
  \footnotesize%
  \begin{tabular}{>{$}c<{$}|*{4}{s}c}
     |\triangulation| & \mcol{c}{\energynorm{p(\param) - p_h(\param)}{\paramFixed}}
                                       & \mcol{c}{\eta_\text{nc}(\cdot; \paramFixed)}
                                                        & \mcol{c}{\eta_\text{r}(\cdot; \param)}
                                                                         & \mcol{c}{\eta_\text{df}(\cdot; \param, \paramHat)}
                                                                                          & \text{eff.} \\ \hline
    \hspace{2.4mm}128 & \sci{3.28}{-1} & \sci{1.66}{-1} & \sci{5.79}{-1} & \sci{3.55}{-1} & 3.36        \\
    \hspace{2.4mm}512 & \sci{1.60}{-1} & \sci{7.89}{-2} & \sci{2.90}{-1} & \sci{1.76}{-1} & 3.40        \\
            2,\hnS048 & \sci{7.78}{-2} & \sci{3.91}{-2} & \sci{1.45}{-1} & \sci{8.73}{-2} & 3.49        \\
            8,\hnS192 & \sci{3.47}{-2} & \sci{1.95}{-2} & \sci{7.27}{-2} & \sci{4.35}{-2} & 3.91        \\ \hline
    \text{order}      & \mcol{c}{1.08} & \mcol{c}{1.03} & \mcol{c}{1.00} & \mcol{c}{1.01} & --
  \end{tabular}
  \caption{%
    Discretization error, estimator components and efficiency of the error estimator for the academic example in Section \ref{subsection::experiments_convergence} with $|\Triangulation| = 1$ and $\param = \paramFixed = \paramHat = 1$.}
  \label{table::experiments::estimator::academic::nonparametric}
\end{table}
\begin{table}
  \centering%
  \footnotesize%
  \begin{tabular}{r||s|sc|ssc}
    \mcol{c}{}               & \mcol{c|}{}     & \multicolumn{2}{c|}{$\paramHat = 1$}
                                                                       & \multicolumn{3}{c}{$\paramHat = 0.1$}  \\\cline{3-7}
    $|\triangulation|\hspace{1.5mm}
                \,/\, \hspace{1.85mm}|\Triangulation|\hspace{1.95mm}$
                             & \mcol{c|}{\eta_\text{r}(\cdot;\param)}
                                               & \mcol{c}{\eta(\cdot;\param, \paramFixed, \paramHat)}
                                                                & eff. & \mcol{c}{\eta_\text{df}(\cdot;\param, \paramHat)}
                                                                                        & \mcol{c}{\eta(\cdot;\param, \paramFixed, \paramHat)}
                                                                                                         & eff. \\\hline\hline
           $128 \,/\, \hspace{1.45mm}2 \times 2\hspace{1.45mm}$
                             & \sci{2.89}{-1}  & \sci{8.10}{-1} & 2.47 & \sci{3.16}{-1} & \sci{7.71}{-1} & 2.35 \\
           $512 \,/\, \hspace{1.45mm}4 \times 4\hspace{1.45mm}$
                             & \sci{7.26}{-2}  & \sci{3.27}{-1} & 2.04 & \sci{1.56}{-1} & \sci{3.08}{-1} & 1.92 \\
    $2,\hnS 048 \,/\, \hspace{1.45mm}8 \times 8\hspace{1.45mm}$
                             & \sci{1.82}{-2}  & \sci{1.45}{-1} & 1.86 & \sci{7.74}{-2} & \sci{1.35}{-1} & 1.73 \\
    $8,\hnS 192 \,/\,               16 \times 16$
                             & \sci{4.54}{-3}  & \sci{6.76}{-2} & 1.95 & \sci{3.85}{-2} & \sci{6.26}{-2} & 1.80 \\\hline
    \mcol{c||}{\text{order}} & \mcol{c|}{2.00} & \mcol{c}{1.20} & --   & \mcol{c}{1.01} & \mcol{c}{1.21} & --
  \end{tabular}
  \caption{%
    Selected estimator components, estimated error and efficiency of the error estimator for the academic example in Section \ref{subsection::experiments_convergence} with $\triangulation$ and $\Triangulation$ simultaneously refined, $\param = 1$ and two choices of $\paramHat$.
    Note that the estimator components $\eta_\text{nc}$ and $\eta_\text{df}$ are not influenced by $\Triangulation$, the estimator components   $\eta_\text{nc}$ and $\eta_\text{r}$ are not influenced by $\paramHat$ and the discretization error is not influenced by either.
    Thus only $\eta_\text{r}$, $\eta$ and its efficiency are given for $\paramHat = 1$ and only $\eta_\text{df}$, $\eta$ and its efficiency are given for $\paramHat = 0.1$ (the other quantities coincide with the ones in Table \ref{table::experiments::estimator::academic::nonparametric}).}
  \label{table::experiments::estimator::academic::offline_online_decomposable}
\end{table}
We choose $\triangulation$ just as in \cite[Section 8.1]{ESV2007} and begin with $\param = \paramFixed = \paramHat = 1$, thus reproducing the nonparametric example studied in \cite[Section 8.1]{ESV2007} (since $\lambda \equiv 1$ and all constants involving $\alpha$ and $\gamma$ are equal to $1$).
For this specific choice of parameters an exact solution is available (see \cite[Section 8.1]{ESV2007}).
In this configuration, $\eta_\text{nc}$ and $\eta_\text{df}$ coincide with their respective counterparts defined in \cite{ESV2007,ESV2010} while $\eta_\text{r}$ is directly influenced by the choice of the coarse triangulation and the parametric nature of $\lambda$ (entering $c_\eps^\Element$).
Choosing $\Triangulation = \Omega$ (the coarse grid configuration with the worst efficiency), we observe results similar to \cite[Table 1]{ESV2007} for $\eta_\text{df}$ and $\eta_\text{nc}$ in Table \ref{table::experiments::estimator::academic::nonparametric}.
In contrast, $\eta_\text{r}$ shows only linear convergence while the residual estimator in \cite[Table 1]{ESV2007} converges with second order.
Overall, the efficiency of the estimator $\eta$ is around $3.5$ (for fixed $|\Triangulation| = 1$) while the efficiency of the estimator in \cite[Table 1]{ESV2007} is around $1.2$.
We can recover the superconvergence of $\eta_\text{r}$, however, by refining $\Triangulation$ along with $\triangulation$ (thus keeping the ratio $H/h$ fixed), see the left columns of Table \ref{table::experiments::estimator::academic::offline_online_decomposable}.
In order to make the estimator off-line/on-line decomposable in the parametric setting, $\paramHat$ has to be fixed throughout the experiment.
Choosing $\paramHat = 0.1$ has no negative impact on the efficiency of the estimator, as we observe in the right columns of Table \ref{table::experiments::estimator::academic::offline_online_decomposable}.
It is often desirable to additionally fix the error norm throughout the experiments.
Choosing $\paramFixed = 0.1$ we still observe a very reasonable efficiency in Table \ref{table::experiments::estimator::academic::fixed_norm}.
\begin{table}
  \centering%
  \footnotesize%
  \begin{tabular}{r|*{3}{s}c}
    $|\triangulation|\hspace{1.5mm}
               \,/\, \hspace{1.85mm}|\Triangulation|\hspace{1.95mm}$
                            & \mcol{c}{\energynorm{p(\param) - p_h(\param)}{\paramFixed}}
                                             & \mcol{c}{\eta_\text{nc}(\cdot; \paramFixed)}
                                                              & \mcol{c}{\eta(\cdot; \param, \paramFixed, \paramHat)}
                                                                               & \text{eff.} \\ \hline
          $128 \,/\, \hspace{1.45mm}2 \times  2\hspace{1.45mm}$
                            & \sci{3.81}{-1} & \sci{1.82}{-1} & \sci{1.18}{0}  & 3.10        \\
          $512 \,/\, \hspace{1.45mm}4 \times  4\hspace{1.45mm}$
                            & \sci{1.87}{-1} & \sci{8.57}{-2} & \sci{5.00}{-1} & 2.67        \\
    $2,\hnS048 \,/\, \hspace{1.45mm}8 \times  8\hspace{1.45mm}$
                            & \sci{9.08}{-2} & \sci{4.22}{-2} & \sci{2.29}{-1} & 2.52        \\
    $8,\hnS192 \,/\,               16 \times 16$
                            & \sci{4.05}{-2} & \sci{2.11}{-2} & \sci{1.10}{-1} & 2.71        \\ \hline
    \mcol{c|}{\text{order}} & \mcol{c}{1.08} & \mcol{c}{1.03} & \mcol{c}{1.14} & --
  \end{tabular}
  \caption{%
      Discretization error, selected estimator component, estimated error and efficiency of the error estimator for the academic example in Section \ref{subsection::experiments_convergence} with $\triangulation$ and $\Triangulation$ simultaneously refined, $\param = 1$ and $\paramFixed = \paramHat = 0.1$.
      Note that $\eta_\text{r}$ and $\eta_\text{df}$ are not influenced by $\paramFixed$ and coincide with Table \ref{table::experiments::estimator::academic::offline_online_decomposable}.}
  \label{table::experiments::estimator::academic::fixed_norm}
\end{table}

\textbf{Multi-scale example.} We consider \eqref{equation::problem::global_pressure} on $\Omega = [0, 5] \times [0, 1]$ with $f(x,y) = 2\cdot 10^3 $ if $(x,y) \in [0.95, 1.10] \times [0.30, 0.45]$, $f(x,y) = -1\cdot 10^3$ if $(x,y) \in [3.00, 3.15] \times [0.75, 0.90]$ or $(x,y) \in [4.25, 4.40] \times [0.25, 0.40]$ and $0$ everywhere else, $\lambda(x,y; \param) = 1 + (1 - \param) \lambda_c(x,y)$, $\kappa_\eps = \kappa \,\text{id}$, homogeneous Dirichlet boundary values and a parameter space $\Params = [0.1, 1]$.
On each $\element \in \triangulation$, $\restrict{\kappa}{\element}$ is the corresponding $0$th entry of the permeability tensor used in the first model of the 10th SPE Comparative Solution Project (which is given by $100 \times 20$ constant tensors, see \cite{spe10}) and $\lambda_c$ models a channel, as depicted in Figure \ref{figure::experiments::estimator::spe10::datafunctions}.
The right hand side $f$ models a strong source in the middle left of the domain and two sinks in the top and right middle of the domain, as is visible in the structure of the solutions (see Figure \ref{figure::experiments::estimator::spe10::datafunctions}).
The role of the parameter $\param$ is to toggle the existence of the channel $\lambda_c$.
Thus $\lambda(\param)\kappa$ reduces to the above mentioned permeability tensor for $\param = 1$ while $\param = 0.1$ models the removal of a large conductivity region near the center of the domain (see the first row in Figure \ref{figure::experiments::estimator::spe10::solutions}).
This missing channel has a visible impact on the structure of the pressure distribution as well as the reconstructed velocities, as we observe in the left column of Figure \ref{figure::experiments::estimator::spe10::solutions}.
With a contrast of $10^6$ in the diffusion tensor and an $\eps$ of about $|\Omega|/2,\hnS000$ this setup is a challenging heterogeneous multi-scale problem.
\begin{figure}[b]
  \centering%
  \footnotesize%
  \includegraphics[width=\textwidth]{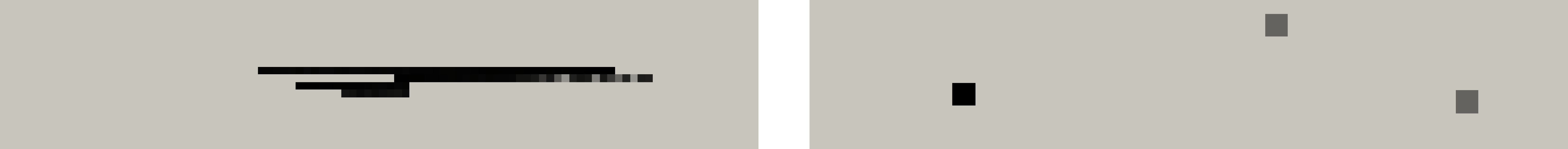}
  \caption{%
    Data functions of the multi-scale example in Section \ref{subsection::experiments_convergence} on a triangulation with $|\triangulation| = 16,\hnS000$ elements: location of the channel function $\lambda_c$ (left) and plot of the force $f$ (right) modeling one source (black: $2\mydot 10^3$) and two sinks (dark gray: $-1\mydot 10^3$, zero elsewhere).}
  \label{figure::experiments::estimator::spe10::datafunctions}
\end{figure}
\begin{figure}
  \centering%
  \footnotesize%
  \includegraphics[width=\textwidth]{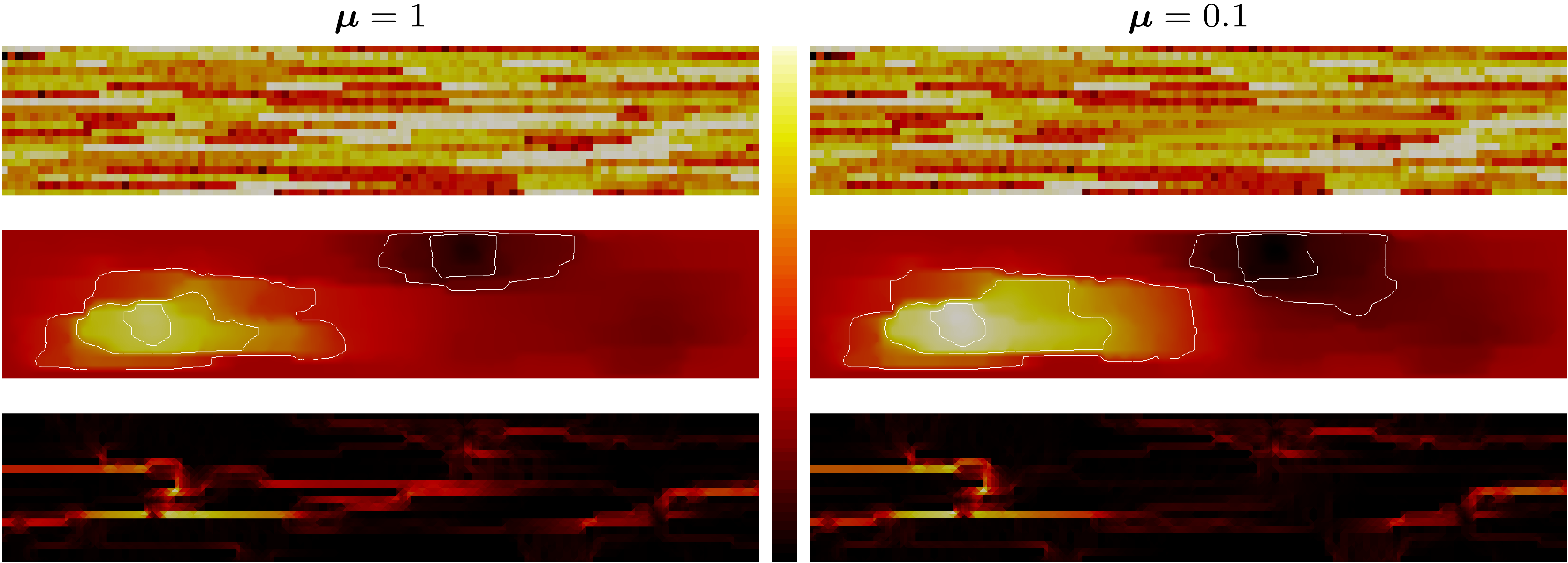}
  \caption{%
    Data functions and sample solutions of the multi-scale example in Section \ref{subsection::experiments_convergence} on a triangulation with $|\triangulation| = 16,\hnS000$ elements for parameters $\param = 1$ (left column) and $\param = 0.1$ (right column).
    In each row both plots share the same color map (middle) with different ranges per row.
    From top to bottom: logarithmic plot of $\lambda(\param)\kappa$ (dark: $1.41\mydot 10^{-3}$, light: $1.41\mydot 10^3$), plot of the pressure $p_h(\param)$ (solution of \eqref{equation::problem::discrete_solution}, dark: $-3.92\mydot 10^{-1}$, light: $7.61\mydot 10{-1}$, isolines at 10\%, 20\%, 45\%, 75\% and 95\%) and plot of the magnitude of the reconstructed diffusive flux $R_h^0[p_h(\param); \param]$ (defined in \eqref{equation::error::flux_reconstruction::1} and \eqref{equation::error::flux_reconstruction::2}, dark: $3.10\mydot 10^{-6}$, light: $3.01\mydot 10^2$).
    Note the presence of high-conductivity channels in the permeability (top left, light regions) throughout large parts of the domain.
    The parameter dependency models a removal of one such channel in the middle right of the domain (top right), well visible in the reconstructed Darcy velocity fields (bottom).}
  \label{figure::experiments::estimator::spe10::solutions}
\end{figure}
\begin{figure}
  \centering%
  \footnotesize%
  \includegraphics[width=\textwidth]{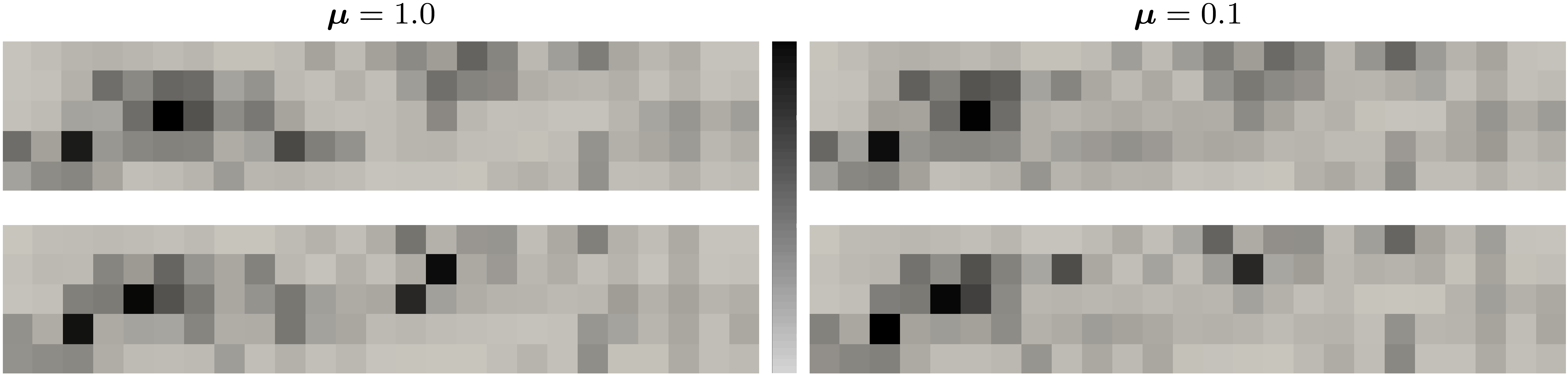}
  \caption{%
    Spatial distribution of the relative error contribution (top), $\energynorm{p(\param) - p_h(\param)}{\paramFixed, \Element} / \energynorm{p(\param) - p_h(\param)}{\paramFixed}$, and the relative estimated error contribution (bottom), $\eta^\Element(p_h(\param); \param, \paramFixed, \paramHat) / \big(\sum_{\Element \in \Triangulation} \eta^\Element(p_h(\param); \param, \paramFixed, \paramHat)^2\big)^{-1/2}$, for all $\Element \in \Triangulation$, for the multi-scale example in Section \ref{subsection::experiments_convergence} with $|\Triangulation| = 25 \times 5$ and $\paramFixed = \paramHat = 0.1$ for parameters $\param = 1$ (left column, light: $2.26\mydot 10^{-3}$, dark: $3.78\mydot 10^{-1}$) and $\param = 0.1$ (right column, light: $4.02\mydot 10^{-3}$, dark: $3.73\mydot 10^{-1}$).}
  \label{figure::experiments::estimator::spe10::error_distribution}
\end{figure}
\begin{table}
  \centering%
  \footnotesize%
  \begin{tabular}{r|*{4}{s}c}
    $|\triangulation|\hspace{3.6mm} \,/\, \hspace{3.35mm}|\Triangulation|\hspace{1.95mm}$
                            & \mcol{c}{\energynorm{p(\param) - p_h(\param)}{\paramFixed}}
                                             & \mcol{c}{\eta_\text{nc}(\cdot; \paramFixed)}
                                                              & \mcol{c}{\eta_\text{r}(\cdot; \param)}
                                                                                & \mcol{c}{\eta_\text{df}(\cdot; \param, \paramHat)}
                                                                                                 & \text{eff.} \\ \hline
           $16,\hnS000 \,/\, \hspace{1.45mm}25 \times 5\hspace{1.45mm}$
                            & \sci{7.49}{-1} & \sci{2.13}{0}  & \sci{1.88}{-9}  & \sci{9.66}{-1} & 4.14        \\
           $64,\hnS000 \,/\, \hspace{1.45mm}50 \times 10$
                            & \sci{4.52}{-1} & \sci{1.46}{0}  & \sci{7.05}{-10} & \sci{6.05}{-1} & 4.58        \\
          $256,\hnS000 \,/\,               100 \times 20$
                            & \sci{2.58}{-1} & \sci{1.02}{0}  & \sci{7.44}{-11} & \sci{3.85}{-1} & 5.44        \\
    $1,\hnS024,\hnS000 \,/\,               200 \times 40$
                            & \sci{1.26}{-1} & \sci{7.20}{-1} & \sci{2.00}{-10} & \sci{2.49}{-1} & 7.70        \\ \hline
    \mcol{c|}{\text{order}} & \mcol{c}{0.86} & \mcol{c}{0.52} & \mcol{c}{1.07}  & \mcol{c}{0.65} & --
  \end{tabular}
  \caption{%
    Discretization error, estimator components and efficiency of the error estimator for the multi-scale example in Section \ref{subsection::experiments_convergence} with $\triangulation$ and $\Triangulation$ simultaneously refined and $\param = \paramFixed = \paramHat = 1$.
    Note that $\eta_\text{r}$ should be close to zero (since $f$ is piecewise constant, compare Figure \ref{figure::experiments::estimator::spe10::datafunctions}) and suffers from numerical inaccuracies (ignoring the last refinement would yield an average order of 2.33 for $\eta_\text{r}$).}
  \label{table::experiments::estimator::spe10::nonparametric}
\end{table}

While we observe in Table \ref{table::experiments::estimator::spe10::nonparametric} that the convergence rates of the estimator components are not as good as in the experiment studied before, the estimator shows an average efficiency of $5.5$, which is quite remarkable considering the contrast of the data functions.
Fixing $\paramFixed = \paramHat = 0.1$ (to obtain a fully off-line/on-line decomposable configuration with a fixed error norm) yields an average efficiency of $10.2$ which is still very reasonable.
In addition we observe a good agreement between the spatial distribution of the error and the estimator indicators in Figure \ref{figure::experiments::estimator::spe10::error_distribution}.

\subsection{Adaptive on-line enrichment}
\label{subsection::experiments_enrichment}

To demonstrate the proposed adaptive on-line enrichment Algorithm \ref{algorithm::online_enrichment} and the flexibility of the LRBMS we study two distinct circumstances.
We first consider a smooth academic example where we disable the greedy procedure and build the reduced bases only by local enrichment.
The second example is again a multi-scale one with global channels in the permeability and we allow for very few global solution snapshots and adaptively enrich afterwards.

For the orthonormalization algorithm \texttt{ONB} in the Greedy Algorithm \ref{algorithm::greedy} we use the stabilized Gram Schmidt procedure implemented in \texttt{pyMOR}\footnote{\url{http://docs.pymor.org/en/0.2.x/_modules/pymor/la/gram_schmidt.html\#gram_schmidt}} with respect to the scalar product given by $(p,q) \mapsto b^\Element(p, q; \paramFixed) + \sum_{\face \in \faces^\Element} b_p^\Element(p, q; \paramFixed)$ on each $\Element \in \Triangulation$.
In contrast to other possible basis extension algorithms (e.g. using a proper orthogonal decomposition) the Gram Schmidt basis extension yields hierarchical local reduced bases.
This is of particular interest in the context of on-line enrichment, since we do not need to update reduced quantities w.r.t existing basis vectors after enrichment.
We always initialize the local reduced bases with the coarse DG basis of order up to one by setting $k_H = 1$ in the Greedy Algorithm \ref{algorithm::greedy}.
We choose the same orthonormalization algorithm in the adaptive basis enrichment Algorithm \ref{algorithm::online_enrichment} and use several marking strategies for \texttt{MARK}, depending on the circumstances (detailed below).
Regarding the overlap for the local enrichment we always choose the overlapping subdomains $\Element_\delta \supset \Element$ to include $\Element$ and all subdomains that touch it, thus choosing an overlap of $\Order(H)$ as motivated in \cite{HP2013}.

Since the error of any reduced solution $\energynorm{p(\param) - p_\red(\param)}{\paramFixed}$ can not be lower than the error of the corresponding detailed solution $\energynorm{p(\param) - p_h(\param)}{\paramFixed}$ we always choose $\Delta_\text{online}$ in Algorithm \ref{algorithm::online_enrichment} to be slightly larger than $\max_{\param \in \Params_\text{online}} \eta(p_h(\param); \param, \paramFixed, \paramHat)$ in our experiments (see below), where $\Params_\text{online} \subset \Params$ is the set of all parameters we consider during the on-line phase.
This is only necessary since we do not allow for an adaptation of $\triangulation$; combining our on-line adaptive LRBMS with the ideas of \cite{Yan2014} would overcome this restriction.

\textbf{Academic example.} We again consider the academic example detailed in Subsection \ref{subsection::experiments_convergence} on fixed triangulations with $|\triangulation| = 131,\hnS072$ and $|\Triangulation| = 8 \times 8$ and choose the set of on-line parameters $\Params_\text{online}$ to consist of 10 randomly chosen parameters $\param_0, \dots, \param_9 \in \Params$.
We set $N_\text{greedy} = 0$ and $k_H = 1$, thus disabling any Greedy search and initializing the local bases with the coarse DG basis of order up to one (consisting of 4 shape functions).
In this setup $\max_{\param \in \Params_\text{online}} \eta(p_h(\param); \param, \paramFixed, \paramHat) = 2.79\cdot 10^{-2}$ and we choose $\Delta_\text{online} = 5\cdot 10^{-2}$ in Algorithm \ref{algorithm::online_enrichment}.

\begin{figure}
  \footnotesize%
  \centering%
  \includegraphics{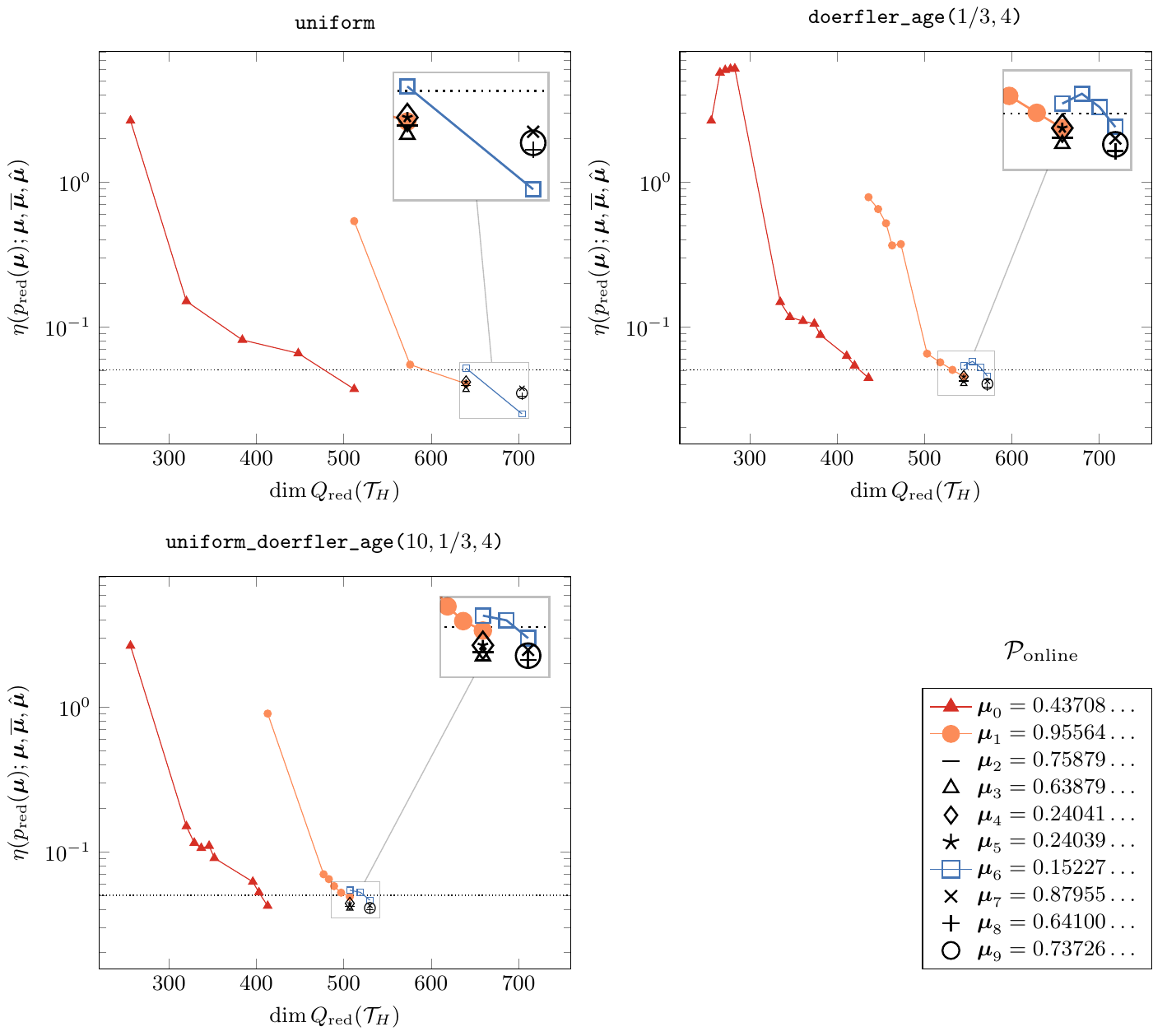}
  \caption{%
    Estimated error evolution during the adaptive on-line phase for the academic example in Section \ref{subsection::experiments_enrichment} with $|\Triangulation| = 64$, $N_\text{greedy} = 0$, $k_H = 1$, $\Delta_\text{online} = 5\cdot 10^{-2}$ (dotted line) and $\paramFixed = \paramHat = 0.1$ for several marking strategies: uniform marking of all subdomains (top left), combined D\"orfler marking with $\theta_\text{doerf} = 1/3$ and age-based marking with $N_\text{age} = 4$ (top right) and additional uniform marking while $\eta(p_\red(\param); \param, \paramFixed, \paramHat) > \theta_\text{uni} \Delta_\text{online}$ with $\theta_\text{uni} = 10$ (bottom left).
    With each strategy the local reduced bases are enriched according to Algorithm \ref{algorithm::online_enrichment} for each subsequently processed on-line parameter $\param_0, \dots, \param_9$ (bottom right).}
  \label{figure::experiments::enrchment::academic::error_evolution}
\end{figure}

We begin by choosing \texttt{MARK} such that $\tildeTriangulation = \Triangulation$ (all coarse elements are marked; denoted by \texttt{uniform} in the following).
For each parameter $\param \in \Params_\text{online}$ this results in a method that is similar to DD methods with overlapping subdomains.
In contrast to traditional DD methods, however, we start with an initial coarse basis and perform a reduced solve before each iteration which helps to spatially spread information.
As we observe in Figure \ref{figure::experiments::enrchment::academic::error_evolution}, top left, it takes four enrichment steps to lower the estimated error for the first on-line parameter $\param_0$ below the desired tolerance and another two enrichment steps for the next on-line parameter $\param_1$ (which is $\max_{\param \in \Params_\text{online}} \param$).
With uniform marking this increases the local basis sizes from four to ten on each coarse element.
The resulting coarse reduced space of dimension 640 is then sufficient to solve for the next four on-line parameters $\param_2, \dots, \param_5$ without enrichment.
One additional enrichment phase is needed for $\param_6$ (which is $\min_{\param \in \Params_\text{online}} \param$) and none for the remaining three on-line parameters.
Note that while this uniform marking strategy may be optimal in the number of enrichment steps it takes to reach the desired error for all on-line parameters it also leads to an unnecessarily high-dimensional coarse reduced space (of $\dim Q_\red(\Triangulation) = 704$) and a high work-load in each enrichment step.

Another popular choice in the context of adaptive mesh refinement is a D\"orfler marking strategy (see \cite{MNS2002} and the references therein), where we collect those coarse elements that contribute most to $\theta_\text{doerf} \sum_{\Element \in \Triangulation} \eta^\Element(\cdot; \param, \paramFixed, \paramHat)^2$ in $\tildeTriangulation \subseteq \Triangulation$, for a given user-dependent parameter $0 < \theta_\text{doerf} \leq 1$.
In addition, similar to \cite{BDD2004,HDO2011},  we count how often each $\Element \in \Triangulation$ was not marked and mark those elements the ``age'' of which is larger than a prescribed $N_\text{age} \in \N$ (resetting the age count of each selected element).
We denote this marking strategy by \texttt{doerfler\_age(}$\theta_\text{doerf}, N_\text{age}$\texttt{)}.
We found that a combination of $\theta_\text{doerf} = 1/3$ and $N_\text{age} = 4$ yielded the smallest overall basis size (of 572), compared to other combinations of $\theta_\text{doerf}$ and $N_\text{age}$ and the \texttt{uniform} marking strategy.
The number of elements marked per step range between five and 52 (over all on-line parameters and all enrichment steps; 23 steps in total) with a mean of 14 and a median of ten.
Of these marked elements between one and 44 have been marked due to their age in 12 of these 23 steps (with an average of 12 and a median of eight, taken over only those 12 steps where elements have been marked due to their age).
We observe in Figure \ref{figure::experiments::enrchment::academic::error_evolution}, top right, that the general behavior of the method with this marking strategy is similar to the one with \texttt{uniform} marking, with some commonalities and differences worth noting.
First of all it also takes three enrichment phases to reach the prescribed error tolerance, and for the same parameters $\param_0$, $\param_1$ and $\param_6$ as above.
But each of these enrichment phases naturally need more steps and large improvements can usually be observed after a lot of elements have been marked due to their age count (see for instance the fifth enrichment step for $\param_0$ or $\param_1$).
In addition we observe that the estimated error for a parameter sometimes increases, in particular in the very beginning (see the first four steps for $\param_0$, the fourth step for $\param_1$ or the first step for $\param_6$).
This is not troublesome since we can only expect a strict improvement in the energy norm induced by the bilinear form that is used in the enrichment.
This shows, never the less, that there is still room for improvement, although using the \texttt{doerfler\_age} marking we could reach a significantly lower overall basis size than using the \texttt{uniform} marking (572 vs. 704).

We propose a combination of the two strategies, namely a uniform marking while the estimated error is far away from the desired tolerance, i.e., $\eta(p_\red(\param); \param, \paramFixed, \paramHat) > \theta_\text{uni} \, \Delta_\text{online}$ for some $\theta_\text{uni} > 0$, followed by a D\"orfler and age-based marking as detailed above.
We denote this marking strategy by \texttt{uniform\_doerfler\_age(}$\theta_\text{uni}, \theta_\text{doerf}, N_\text{age}$\texttt{)}.
As we observe in Figure \ref{figure::experiments::enrchment::academic::error_evolution}, bottom left, this marking strategy combines advantages of both previous approaches, recovering the rapid error decrease of the \texttt{uniform} marking strategy far away from the desired tolerance (see the first step for $\param_0$ and $\param_1$) while yielding the smallest overall basis size of 530 (using a factor of $\theta_\text{uni} = 10$) due to the \texttt{doerfler\_age} marking strategy.
The smoothness and symmetry of the problem is reflected in the spatial distribution of the final local basis sizes (see Figure \ref{figure::experiments::enrichment::academic::final_local_basis_sizes}).

\begin{SCfigure}[50]
  \footnotesize%
  \includegraphics[height=28mm,trim=39 28 499 76, clip]{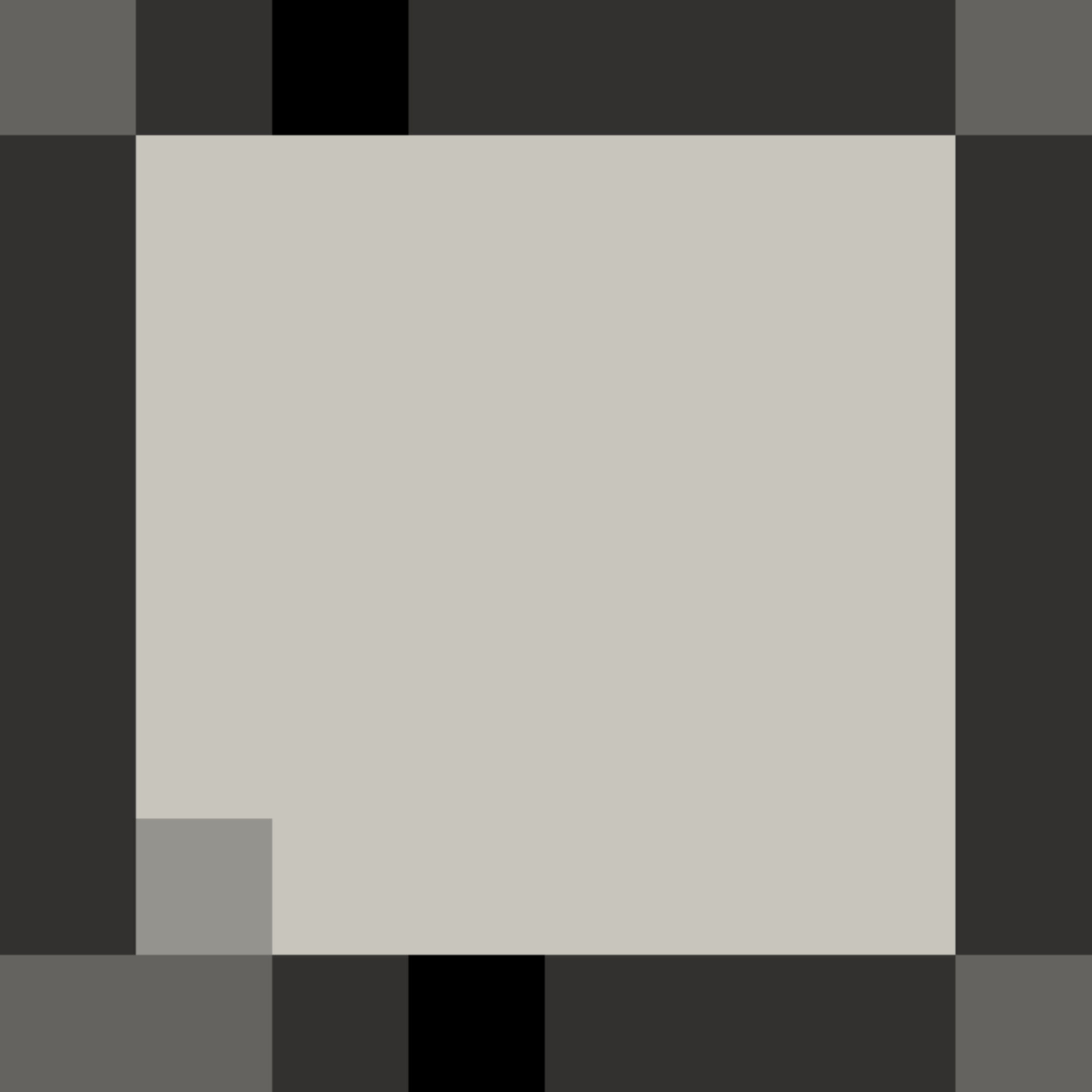}
  \caption{%
    Spatial distribution of the final sizes of the local reduced bases, $|\varPhi^\Element|$ (light: 7, dark: 11), for all $\Element \in \Triangulation$ after the adaptive on-line phase for the academic example in Section \ref{subsection::experiments_enrichment} with $\Omega = [-1, 1]^2$, $|\Triangulation| = 8 \times 8$ and the \textnormal{\texttt{uniform\_doerfler\_age(}$10, 1/3, 4$\texttt{)}} marking strategy (see Figure \ref{figure::experiments::enrchment::academic::error_evolution}, bottom left).}
  \label{figure::experiments::enrichment::academic::final_local_basis_sizes}
\end{SCfigure}

\textbf{Multi-scale example.} We again consider the multi-scale example detailed in \S \ref{subsection::experiments_convergence} on fixed triangulations with $|\triangulation| = 1,\hnS014,\hnS000$ and $|\Triangulation| = 25 \times 5$ and choose the set of on-line parameters $\Params_\text{online}$ to consist of the same 10 randomly chosen parameters $\param_0, \dots, \param_9 \in \Params$ as in the previous example.
In this setup $\max_{\param \in \Params_\text{online}} \eta(p_h(\param); \param, \paramFixed, \paramHat) = 2.66$ and we choose $\Delta_\text{online} = 2$ in Algorithm \ref{algorithm::online_enrichment}.
\begin{figure}
  \centering%
  \footnotesize%
  \includegraphics{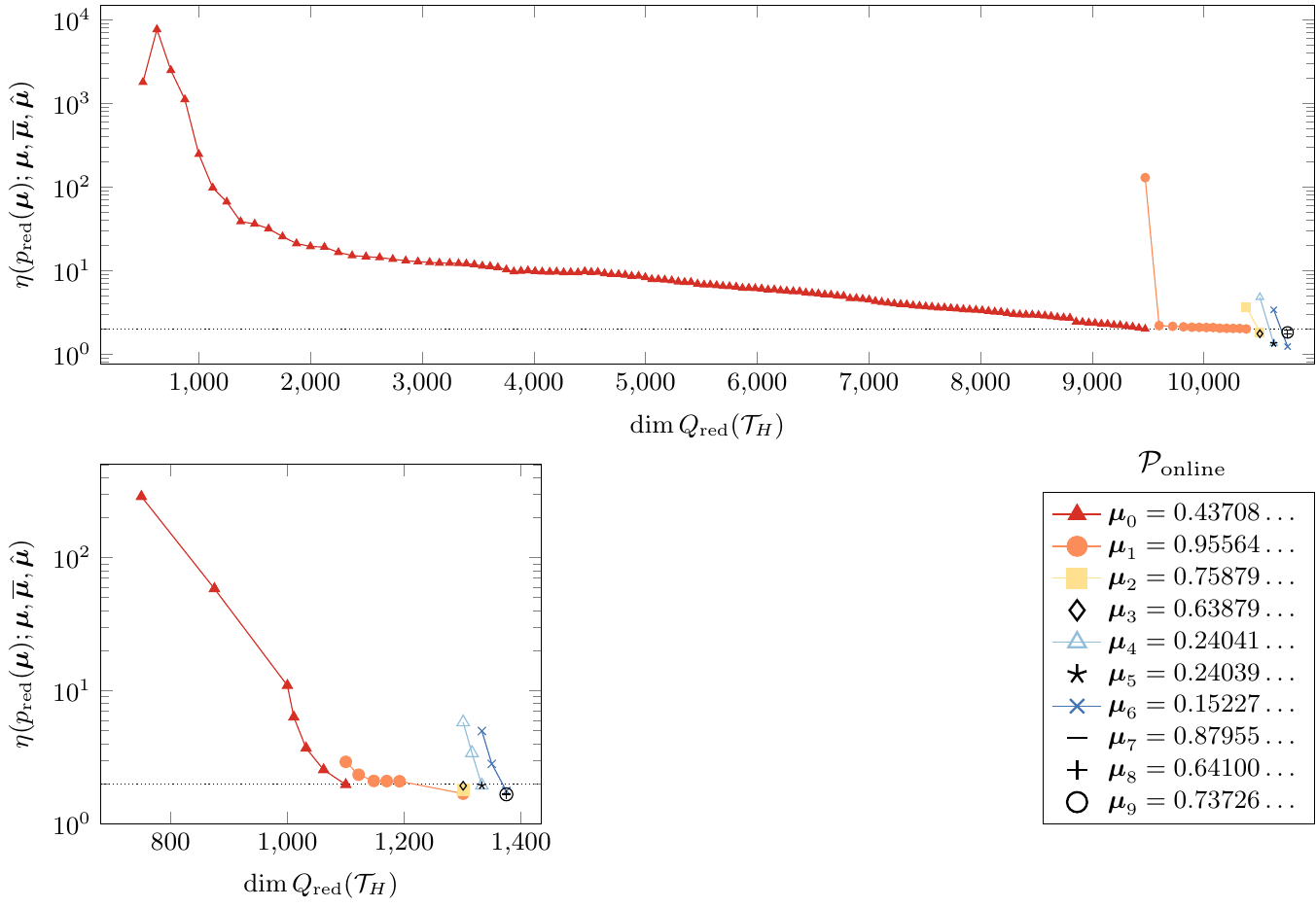}
  \caption{%
    Estimated error evolution during the adaptive on-line phase for the multi-scale example in \S \ref{subsection::experiments_enrichment} with $|\Triangulation| = 125$, $k_H = 1$, $\Delta_\text{online} = 2$ (dotted line), $\paramFixed = \paramHat = 0.1$ for different on-line and off-line strategies:
    no global snapshot (Greedy search disabled, $N_\text{greedy} = 0$) during the off-line phase, uniform marking during the on-line phase (top) and two global snapshots (Greedy search on $\Params_\text{train} = \{0.1, 1\}$, $N_\text{greedy} = 2$) and combined uniform marking while $\eta(p_\red(\param); \param, \paramFixed, \paramHat) > \theta_\text{uni} \Delta_\text{online}$ with $\theta_\text{uni} = 10$, D\"orfler marking with $\theta_\text{doerf} = 0.85$ and age-based marking with $N_\text{age} = 4$ (bottom left); note the different scales.
    With each strategy the local reduced bases are enriched according to Algorithm \ref{algorithm::online_enrichment} while subsequently processing the on-line parameters $\param_0, \dots, \param_9$ (bottom right).}
  \label{figure::experiments::enrichment::spe10::error_evolution}
\end{figure}

We first set $N_\text{greedy} = 0$ and $k_H = 1$ (thus disabling any Greedy search in the off-line phase and initializing the local bases with the coarse DG basis of order up to one); in the on-line phase we use the \texttt{uniform} marking strategy (see above).
As we observe in Figure \ref{figure::experiments::enrichment::spe10::error_evolution}, top, it takes 129 enrichment steps to lower the estimated error below the desired tolerance for the first on-line parameter $\param_0$.
This is not surprising since the data functions exhibit strong multi-scale features and non-local high-conductivity channels connecting domain boundaries (see Figure \ref{figure::experiments::estimator::spe10::solutions}).
After this extensive enrichment it takes 12 steps for $\param_1$ and none or one enrichment steps to reach the desired tolerance for the other on-line parameters.
The resulting coarse reduced space is of size $10,\hnS749$ (with an average of 86 basis functions per subdomain), which is clearly not optimal.
Although each subdomain was marked for enrichment, the sizes of the final local reduced bases may differ since the local Gram Schmidt basis extension may reject updates (if the added basis function is locally not linearly independent).
As we observe in Figure \ref{figure::experiments::enrichment::spe10::final_local_basis_sizes}, left, this is indeed the case with local basis sizes ranging between 24 and 148.

To remedy the situation we allow for two global snapshots during the off-line phase (setting $N_\text{greedy} = 2$, $\Params_\text{train} = \{0.1, 1\}$) and use the adaptive \texttt{uniform\_doerfler\_age} marking strategy (see above) in the on-line phase.
With two global solution snapshots incorporated in the basis the situation improves significantly, as we observe in Figure \ref{figure::experiments::enrichment::spe10::error_evolution}, bottom left, and there is no qualitative difference of the evolution of the estimated error during the adaptive on-line phase between the academic example studied above and this highly heterogeneous multi-scale example (compare Figure \ref{figure::experiments::enrchment::academic::error_evolution}, bottom left).
In total we observe only two enrichment steps with uniform marking (see the first two step for $\param_0$).
The number of elements marked range between 11 and 110 (over all on-line parameters and all but the first two enrichment steps) with a mean of 29 and a median of 22.
Of these marked elements only once have 87 out of 110 elements been marked due to their age (see the last step for $\param_1$).
Overall we could reach a significantly lower overall basis size than in the previous setup ($1,\hnS375$ vs. $10,\hnS749$) and the sizes of the final local bases range between only nine and 2 (compared to 24 to 148 above).
We also observe in Figure \ref{figure::experiments::enrichment::spe10::final_local_basis_sizes}, right, that the spatial distribution of the basis sizes follows the spatial structure of the data functions (compare Figures \ref{figure::experiments::estimator::spe10::datafunctions}, \ref{figure::experiments::estimator::spe10::solutions}), which nicely shows the localization qualities of out error estimator.
\begin{figure}
  \footnotesize%
  \centering%
  \includegraphics[width=\textwidth]{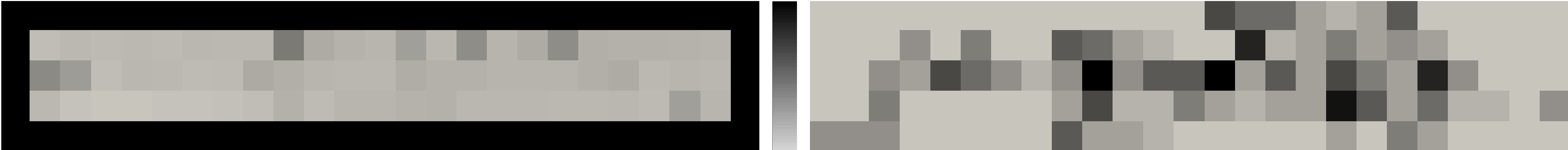}
  \caption{%
    Spatial distribution of the final sizes of the local reduced bases, $|\varPhi^\Element|$ for all $\Element \in \Triangulation$, after the adaptive online phase for the multi-scale example in \S \ref{subsection::experiments_enrichment} with $\Omega = [0, 5]\times[0, 1]$, $|\Triangulation| = 25 \times 5$ for the two strategies shown in Figure \ref{figure::experiments::enrichment::spe10::error_evolution}: no global snapshot with uniform enrichment (left, light: 24, dark: 148) and two global snapshots with adaptive enrichment (right, light: 9, dark: 20).
    Note the pronounced structure (right) reflecting the spatial structure of the data functions (compare Figures \ref{figure::experiments::estimator::spe10::datafunctions} and \ref{figure::experiments::estimator::spe10::solutions}).}
  \label{figure::experiments::enrichment::spe10::final_local_basis_sizes}
\end{figure}

\section{Conclusion}

In this contribution we equipped the localized Reduced Basis multi-scale method with a rigorous and efficient localized a-posteriori error estimator and a new adaptive on-line enrichment procedure.
The LRBMS method was originally introduced in \cite{KOH2011,AHKO2012} with a standard residual based a-posteriori error estimator for the reduced solution against the high-dimensional solution that could not give local error information and was computationally very costly during the off-line phase of the computation.
In addition the method was restricted to the SWIPDG discretization.
We extended the original LRBMS method to locally allow for arbitrary discretizations of at least first order within each subdomain.
In addition, we proposed a new error estimator for the high-dimensional as well as the reduced solution against the weak solution in the spirit of \cite{ESV2007,ESV2010}, based on a local conservative flux reconstruction.
The proposed estimator yields a guaranteed upper bound (involving no unknown constants), it is locally efficient and can be computed using local information only.
The estimator is efficiently off-line/on-line decomposable and allows to estimate the error in parameter dependent norms.

Using the local information of the estimator we proposed a new adaptive on-line enrichment strategy, where we extend the local reduced bases by solutions of local corrector problems posed on overlapping subdomains.
This is done during the on-line phase, thus deviating from the strict off-line/on-line separation of traditional RB methods, but only on those subdomains that have been selected by the estimator.
This strategy allows to guarantee the quality of the reduced solution during the on-line phase, even if an insufficient reduced basis had been prepared in the off-line phase (in contrast to traditional RB methods where the reduced basis is fixed for the whole on-line phase).

We provide numerical experiments to demonstrate the performance of the proposed estimator in the parametric setting for the high-dimensional discretization for an academic as well as a highly heterogeneous multi-scale example, where the estimator proves to be very efficient (see Section \ref{subsection::experiments_convergence}).
We also demonstrate the performance of the newly proposed adaptive on-line enrichment strategy for both examples.
It is particular noteworthy that only very few global solution snapshots were needed during the off-line phase for the multi-scale example to sufficiently prepare a reduced basis that was then adaptively enriched during the off-line phase (see Section \ref{subsection::experiments_enrichment}).
For the academic example of smooth data functions without any multi-scale features no global solution snapshots were required at all.

First findings regarding the new estimator have been published in \cite{OS2014} while first steps in the direction of on-line enrichment have been discussed in \cite{AO2013}.


\begin{thebibliography}{10}


\providecommand{\url}[1]{\texttt{#1}}
\expandafter\ifx\csname urlstyle\endcsname\relax
  \providecommand{\doi}[1]{doi: #1}\else
  \providecommand{\doi}{doi: \begingroup \urlstyle{rm}\Url}\fi

\bibitem[1]{dunepdelab}
\emph{dune-pdelab}.
\newblock \url{http://www.dune-project.org/pdelab/index.html}.
\newblock \,Version:\,February 2014

\bibitem[2]{Abd2012}
\textsc{Abdulle}, A.:
\newblock Discontinuous {G}alerkin finite element heterogeneous multiscale
  method for elliptic problems with multiple scales.
\newblock {In: }\emph{Math. Comp.} 81 (2012), Nr. 278, 687--713.
\newblock \url{http://dx.doi.org/10.1090/S0025-5718-2011-02527-5}. --
\newblock DOI 10.1090/S0025--5718--2011--02527--5. --
\newblock ISSN 0025--5718

\bibitem[3]{AB2013}
\textsc{Abdulle}, A. ; \textsc{Bai}, Y.:
\newblock Adaptive reduced basis finite element heterogeneous multiscale
  method.
\newblock {In: }\emph{Comput. Methods Appl. Mech. Engrg.} 257 (2013), 203--220.
\newblock \url{http://dx.doi.org/10.1016/j.cma.2013.01.002}. --
\newblock DOI 10.1016/j.cma.2013.01.002. --
\newblock ISSN 0045--7825

\bibitem[4]{AB2014}
\textsc{Abdulle}, A. ; \textsc{Bai}, Y.:
\newblock Reduced-order modelling numerical homogenization.
\newblock {In: }\emph{Philos. Trans. R. Soc. Lond. Ser. A Math. Phys. Eng.
  Sci.} 372 (2014), Nr. 2021, 20130388.
\newblock \url{http://dx.doi.org/10.1098/rsta.2013.0388}. --
\newblock DOI 10.1098/rsta.2013.0388. --
\newblock ISSN 1364--503X

\bibitem[5]{AEEV2012}
\textsc{Abdulle}, A. ; \textsc{E}, W. ; \textsc{Engquist}, B.  ;
  \textsc{Vanden-Eijnden}, E.:
\newblock The heterogeneous multiscale method.
\newblock {In: }\emph{Acta Numer.} 21 (2012), S. 1--87

\bibitem[6]{AH2014a}
\textsc{Abdulle}, A. ; \textsc{Henning}, P.:
\newblock A reduced basis localized orthogonal decomposition.
\newblock {In: }\emph{ArXiv e-prints}  (2014).
\newblock \url{http://arxiv.org/abs/1410.3253}

\bibitem[7]{AH2014}
\textsc{Abdulle}, A. ; \textsc{Huber}, M.~E.:
\newblock Discontinuous {G}alerkin finite element heterogeneous multiscale
  method for advection-diffusion problems with multiple scales.
\newblock {In: }\emph{Numer. Math.} 126 (2014), Nr. 4, S. 589--633. --
\newblock ISSN 0029--599X

\bibitem[8]{AHKO2012}
\textsc{Albrecht}, F. ; \textsc{Haasdonk}, B. ; \textsc{Kaulmann}, S.  ;
  \textsc{Ohlberger}, M.:
\newblock The localized reduced basis multiscale method.
\newblock {In: }\emph{Proceedings of Algoritmy 2012, Conference on Scientific
  Computing, Vysoke Tatry, Podbanske, September 9-14, 2012}, Slovak University
  of Technology in Bratislava, Publishing House of STU, 2012. --
\newblock ISSN 978--80--227--3742--5, S. 393--403

\bibitem[9]{AO2013}
\textsc{Albrecht}, F. ; \textsc{Ohlberger}, M.:
\newblock The localized reduced basis multi-scale method with online
  enrichment.
\newblock {In: }\emph{Oberwolfach Rep.} 7 (2013), S. 12--15.
\newblock \url{http://dx.doi.org/10.4171/OWR/2013/07}. --
\newblock DOI 10.4171/OWR/2013/07

\bibitem[10]{BMNP2004}
\textsc{Barrault}, M. ; \textsc{Maday}, Y. ; \textsc{Nguyen}, N.~C.  ;
  \textsc{Patera}, A.~T.:
\newblock An `empirical interpolation' method: application to efficient
  reduced-basis discretization of partial differential equations.
\newblock {In: }\emph{C. R. Math. Acad. Sci. Paris} 339 (2004), Nr. 9,
  667--672.
\newblock \url{http://dx.doi.org/10.1016/j.crma.2004.08.006}. --
\newblock DOI 10.1016/j.crma.2004.08.006. --
\newblock ISSN 1631--073X

\bibitem[11]{BBD+2008a}
\textsc{Bastian}, P. ; \textsc{Blatt}, M. ; \textsc{Dedner}, A. ;
  \textsc{Engwer}, C. ; \textsc{Kl{\"o}fkorn}, R. ; \textsc{Kornhuber}, R. ;
  \textsc{Ohlberger}, M.  ; \textsc{Sander}, O.:
\newblock A generic grid interface for parallel and adaptive scientific
  computing. {II}. {I}mplementation and tests in {DUNE}.
\newblock {In: }\emph{Computing} 82 (2008), Nr. 2-3, 121--138.
\newblock \url{http://dx.doi.org/10.1007/s00607-008-0004-9}. --
\newblock DOI 10.1007/s00607--008--0004--9. --
\newblock ISSN 0010--485X

\bibitem[12]{BBD+2008}
\textsc{Bastian}, P. ; \textsc{Blatt}, M. ; \textsc{Dedner}, A. ;
  \textsc{Engwer}, C. ; \textsc{Kl{\"o}fkorn}, R. ; \textsc{Ohlberger}, M.  ;
  \textsc{Sander}, O.:
\newblock A generic grid interface for parallel and adaptive scientific
  computing. {I}. {A}bstract framework.
\newblock {In: }\emph{Computing} 82 (2008), Nr. 2-3, 103--119.
\newblock \url{http://dx.doi.org/10.1007/s00607-008-0003-x}. --
\newblock DOI 10.1007/s00607--008--0003--x. --
\newblock ISSN 0010--485X

\bibitem[13]{BDD2004}
\textsc{Binev}, P. ; \textsc{Dahmen}, W.  ; \textsc{DeVore}, R.:
\newblock Adaptive finite element methods with convergence rates.
\newblock {In: }\emph{Numer. Math.} 97 (2004), Nr. 2, 219--268.
\newblock \url{http://dx.doi.org/10.1007/s00211-003-0492-7}. --
\newblock DOI 10.1007/s00211--003--0492--7. --
\newblock ISSN 0029--599X

\bibitem[14]{Boy2008}
\textsc{Boyaval}, S.:
\newblock Reduced-basis approach for homogenization beyond the periodic
  setting.
\newblock {In: }\emph{Multiscale Model. Simul.} 7 (2008), Nr. 1, 466--494.
\newblock \url{http://dx.doi.org/10.1137/070688791}. --
\newblock DOI 10.1137/070688791. --
\newblock ISSN 1540--3459

\bibitem[15]{colorbrewer}
\textsc{Brewer}, C.~A.:
\newblock \emph{ColorBrewer}.
\newblock \url{http://colorbrewer2.org/}.
\newblock \,Version:\,January 2015

\bibitem[16]{BZ2006}
\textsc{Burman}, E. ; \textsc{Zunino}, P.:
\newblock A domain decomposition method based on weighted interior penalties
  for advection-diffusion-reaction problems.
\newblock {In: }\emph{SIAM J. Numer. Anal.} 44 (2006), Nr. 4, 1612--1638
  (electronic).
\newblock \url{http://dx.doi.org/10.1137/050634736}. --
\newblock DOI 10.1137/050634736. --
\newblock ISSN 0036--1429

\bibitem[17]{CHM2006}
\textsc{Chen}, Z. ; \textsc{Huan}, G.  ; \textsc{Ma}, Y.:
\newblock \emph{Computational methods for multiphase flows in porous media}.
  Bd.~2.
\newblock Siam, 2006

\bibitem[18]{DKN2014}
\textsc{Dedner}, A. ; \textsc{Kl{\"o}fkorn}, R.  ; \textsc{Nolte}, M.:
\newblock The {DUNE}-{ALUG}rid {M}odule.
\newblock {In: }\emph{ArXiv e-prints}  (2014).
\newblock \url{http://arxiv.org/abs/1407.6954}

\bibitem[19]{DKNO2010}
\textsc{Dedner}, A. ; \textsc{Kl{\"o}fkorn}, R. ; \textsc{Nolte}, M.  ;
  \textsc{Ohlberger}, M.:
\newblock A generic interface for parallel and adaptive discretization schemes:
  abstraction principles and the {DUNE}-{FEM} module.
\newblock {In: }\emph{Computing} 90 (2010), Nr. 3-4, 165--196.
\newblock \url{http://dx.doi.org/10.1007/s00607-010-0110-3}. --
\newblock DOI 10.1007/s00607--010--0110--3. --
\newblock ISSN 0010--485X

\bibitem[20]{DHO2012}
\textsc{Drohmann}, M. ; \textsc{Haasdonk}, B.  ; \textsc{Ohlberger}, M.:
\newblock Reduced Basis Approximation for Nonlinear Parametrized Evolution
  Equations based on Empirical Operator Interpolation.
\newblock {In: }\emph{SIAM J. Sci. Comput.} 34 (2012), S. A937--A969.
\newblock \url{http://dx.doi.org/10.1137/10081157X}. --
\newblock DOI 10.1137/10081157X. --
\newblock ISSN 1064--8275

\bibitem[21]{EGH2013}
\textsc{Efendiev}, Y. ; \textsc{Galvis}, J.  ; \textsc{Hou}, T.~Y.:
\newblock Generalized multiscale finite element methods ({GM}s{FEM}).
\newblock {In: }\emph{J. Comput. Phys.} 251 (2013), S. 116--135. --
\newblock ISSN 0021--9991

\bibitem[22]{EH2009}
\textsc{Efendiev}, Y. ; \textsc{Hou}, T.~Y.:
\newblock \emph{Surveys and Tutorials in the Applied Mathematical Sciences}.
  Bd.~4: {\emph{Multiscale finite element methods}}.
\newblock New York : Springer, 2009. --
\newblock  xii+234 S. --
\newblock ISBN 978--0--387--09495--3. --
\newblock Theory and applications

\bibitem[23]{ER2007}
\textsc{Epshteyn}, Y. ; \textsc{Rivi{\`e}re}, B.:
\newblock Estimation of penalty parameters for symmetric interior penalty
  Galerkin methods.
\newblock {In: }\emph{J. Comput. Appl. Math.} 206 (2007), Nr. 2, S. 843--872

\bibitem[24]{ES2008}
\textsc{Ern}, A. ; \textsc{Stephansen}, A.~F.:
\newblock A posteriori energy-norm error estimates for advection-diffusion
  equations approximated by weighted interior penalty methods.
\newblock {In: }\emph{J. Comput. Math.} 26 (2008), Nr. 4, S. 488--510. --
\newblock ISSN 0254--9409

\bibitem[25]{ESV2007}
\textsc{Ern}, A. ; \textsc{Stephansen}, A.~F.  ; \textsc{Vohral{\'{\i}}k}, M.:
\newblock Improved energy norm a posteriori error estimation based on flux
  reconstruction for discontinuous Galerkin methods.
\newblock {In: }\emph{Preprint R07050, Laboratoire Jacques-Louis Lions \& HAL
  Preprint} 193540 (2007)

\bibitem[26]{ESV2010}
\textsc{Ern}, A. ; \textsc{Stephansen}, A.~F.  ; \textsc{Vohral{\'\i}k}, M:
\newblock Guaranteed and robust discontinuous Galerkin a posteriori error
  estimates for convection--diffusion--reaction problems.
\newblock {In: }\emph{J. Comput. Appl. Math.} 234 (2010), Nr. 1, S. 114--130

\bibitem[27]{ESZ2009}
\textsc{Ern}, A. ; \textsc{Stephansen}, A.~F.  ; \textsc{Zunino}, P.:
\newblock A discontinuous Galerkin method with weighted averages for
  advection--diffusion equations with locally small and anisotropic
  diffusivity.
\newblock {In: }\emph{IMA J. Numer. Anal.} 29 (2009), Nr. 2, S. 235--256

\bibitem[28]{Feu}
\textsc{Feuers\"anger}, C.:
\newblock \emph{Manual for Package pgfplots},
  \url{http://sourceforge.net/projects/pgf/}

\bibitem[29]{HDO2011}
\textsc{Haasdonk}, B. ; \textsc{Dihlmann}, M.  ; \textsc{Ohlberger}, M.:
\newblock A training set and multiple bases generation approach for
  parameterized model reduction based on adaptive grids in parameter space.
\newblock {In: }\emph{Math. Comput. Model. Dyn. Syst.} 17 (2011), Nr. 4,
  423--442.
\newblock \url{http://dx.doi.org/10.1080/13873954.2011.547674}. --
\newblock DOI 10.1080/13873954.2011.547674. --
\newblock ISSN 1387--3954

\bibitem[30]{HJ2011}
\textsc{Hajibeygi}, H. ; \textsc{Jenny}, P.:
\newblock Adaptive iterative multiscale finite volume method.
\newblock {In: }\emph{J. Comput. Phys.} 230 (2011), Nr. 3, 628--643.
\newblock \url{http://dx.doi.org/10.1016/j.jcp.2010.10.009}. --
\newblock DOI 10.1016/j.jcp.2010.10.009. --
\newblock ISSN 0021--9991

\bibitem[31]{HO2009}
\textsc{Henning}, P. ; \textsc{Ohlberger}, M.:
\newblock The heterogeneous multiscale finite element method for elliptic
  homogenization problems in perforated domains.
\newblock {In: }\emph{Numer. Math.} 113 (2009), Nr. 4, 601--629.
\newblock \url{http://dx.doi.org/10.1007/s00211-009-0244-4}. --
\newblock DOI 10.1007/s00211--009--0244--4. --
\newblock ISSN 0029--599X

\bibitem[32]{HP2013}
\textsc{Henning}, P. ; \textsc{Peterseim}, D.:
\newblock Oversampling for the multiscale finite element method.
\newblock {In: }\emph{Multiscale Model. Simul.} 11 (2013), Nr. 4, 1149--1175.
\newblock \url{http://dx.doi.org/10.1137/120900332}. --
\newblock DOI 10.1137/120900332. --
\newblock ISSN 1540--3459

\bibitem[33]{IQR2012}
\textsc{Iapichino}, L. ; \textsc{Quarteroni}, A.  ; \textsc{Rozza}, G.:
\newblock A reduced basis hybrid method for the coupling of parametrized
  domains represented by fluidic networks.
\newblock {In: }\emph{Comput. Methods Appl. Mech. Engrg.} 221/222 (2012),
  63--82.
\newblock \url{http://dx.doi.org/10.1016/j.cma.2012.02.005}. --
\newblock DOI 10.1016/j.cma.2012.02.005. --
\newblock ISSN 0045--7825

\bibitem[34]{IQRV2014}
\textsc{Iapichino}, L. ; \textsc{Quarteroni}, A. ; \textsc{Rozza}, G.  ;
  \textsc{Volkwein}, S.:
\newblock Reduced basis method for the {S}tokes equations in decomposable
  parametrized domains using greedy optimization.
\newblock {In: }\emph{{ECMI} 2014 proceedings}.
\newblock Heildeberg : Springer, 2014 (ECMI book subseries of Mathematics in
  Industry), 1--7. --
\newblock EPFL MATHICSE Report 28.2014

\bibitem[35]{KP2003}
\textsc{Karakashian}, O.~A. ; \textsc{Pascal}, F.:
\newblock A posteriori error estimates for a discontinuous {G}alerkin
  approximation of second-order elliptic problems.
\newblock {In: }\emph{SIAM J. Numer. Anal.} 41 (2003), Nr. 6, 2374--2399
  (electronic).
\newblock \url{http://dx.doi.org/10.1137/S0036142902405217}. --
\newblock DOI 10.1137/S0036142902405217. --
\newblock ISSN 0036--1429

\bibitem[36]{KOH2011}
\textsc{Kaulmann}, S. ; \textsc{Ohlberger}, M.  ; \textsc{Haasdonk}, B.:
\newblock A new local reduced basis discontinuous {G}alerkin approach for
  heterogeneous multiscale problems.
\newblock {In: }\emph{C. R. Math. Acad. Sci. Paris} 349 (2011), Nr. 23-24,
  1233--1238.
\newblock \url{http://dx.doi.org/10.1016/j.crma.2011.10.024}. --
\newblock DOI 10.1016/j.crma.2011.10.024. --
\newblock ISSN 1631--073X

\bibitem[37]{LM2007}
\textsc{Larson}, M.~G. ; \textsc{M{\aa}lqvist}, A.:
\newblock Adaptive variational multiscale methods based on a posteriori error
  estimation: energy norm estimates for elliptic problems.
\newblock {In: }\emph{Comput. Methods Appl. Mech. Engrg.} 196 (2007), Nr.
  21-24, 2313--2324.
\newblock \url{http://dx.doi.org/10.1016/j.cma.2006.08.019}. --
\newblock DOI 10.1016/j.cma.2006.08.019. --
\newblock ISSN 0045--7825

\bibitem[38]{LMR2007}
\textsc{L{\o}vgren}, A.~E. ; \textsc{Maday}, Y.  ; \textsc{R{\o}nquist}, E.~M.:
\newblock The reduced basis element method for fluid flows.
\newblock \,Version:\,2007.
\newblock \url{http://dx.doi.org/10.1007/978-3-7643-7742-7{\_}8}.
\newblock {In: }\emph{Analysis and simulation of fluid dynamics}.
\newblock Basel : Birkh\"auser, 2007 (Adv. Math. Fluid Mech.). --
\newblock DOI 10.1007/978--3--7643--7742--7{\_}8, 129--154

\bibitem[39]{MR2002}
\textsc{Maday}, Y. ; \textsc{R{\o}nquist}, E.~M.:
\newblock A reduced-basis element method.
\newblock {In: }\emph{C. R. Math. Acad. Sci. Paris} 335 (2002), Nr. 2,
  195--200.
\newblock \url{http://dx.doi.org/10.1016/S1631-073X(02)02427-5}. --
\newblock DOI 10.1016/S1631--073X(02)02427--5. --
\newblock ISSN 1631--073X

\bibitem[40]{pymor}
\textsc{Milk}, R. ; \textsc{Rave}, S.  ; \textsc{Schindler}, F.:
\newblock \emph{py{MOR}, {M}odel {O}rder {R}eduction with {P}ython}.
\newblock \url{http://dx.doi.org/10.5281/zenodo.13897}

\bibitem[41]{dunestuff}
\textsc{Milk}, R. ; \textsc{Schindler}, F.:
\newblock \emph{dune-stuff}.
\newblock \url{http://dx.doi.org/10.5281/zenodo.13891}

\bibitem[42]{MNS2002}
\textsc{Morin}, P. ; \textsc{Nochetto}, R.~H.  ; \textsc{Siebert}, K.~G.:
\newblock Convergence of adaptive finite element methods.
\newblock {In: }\emph{SIAM Rev.} 44 (2002), Nr. 4, 631--658 (2003).
\newblock \url{http://dx.doi.org/10.1137/S0036144502409093}. --
\newblock DOI 10.1137/S0036144502409093. --
\newblock ISSN 0036--1445. --
\newblock Revised reprint of ``Data oscillation and convergence of adaptive
  FEM'' [SIAM J. Numer. Anal. {{\bf{3}}8} (2000), no. 2, 466--488 (electronic);
  MR1770058 (2001g:65157)]

\bibitem[43]{Nor2008}
\textsc{Nordbotten}, J.~M.:
\newblock Adaptive variational multiscale methods for multiphase flow in porous
  media.
\newblock {In: }\emph{Multiscale Model. Simul.} 7 (2008), Nr. 3, 1455--1473.
\newblock \url{http://dx.doi.org/10.1137/080724745}. --
\newblock DOI 10.1137/080724745. --
\newblock ISSN 1540--3459

\bibitem[44]{Ohl2012}
\textsc{Ohlberger}, M.:
\newblock Error control based model reduction for multiscale problems.
\newblock {In: }\emph{Proceedings of Algoritmy 2012, Conference on Scientific
  Computing, Vysoke Tatry, Podbanske, September 9-14, 2012}, Slovak University
  of Technology in Bratislava, Publishing House of STU, 2012. --
\newblock ISSN 978--80--227--3742--5, S. 1--10

\bibitem[45]{OS2014}
\textsc{Ohlberger}, M. ; \textsc{Schindler}, F.:
\newblock A-Posteriori Error Estimates for the Localized Reduced Basis
  Multi-Scale Method.
\newblock \,Version:\,2014.
\newblock \url{http://dx.doi.org/10.1007/978-3-319-05684-5\_41}.
\newblock {In: }\textsc{Fuhrmann}, J{\"u}rgen (Hrsg.) ; \textsc{Ohlberger},
  Mario (Hrsg.)  ; \textsc{Rohde}, Christian (Hrsg.): \emph{Finite Volumes for
  Complex Applications VII-Methods and Theoretical Aspects} Bd.~77.
\newblock Springer International Publishing, 2014. --
\newblock DOI 10.1007/978--3--319--05684--5\_41. --
\newblock ISBN 978--3--319--05683--8, 421-429

\bibitem[46]{PVWW2013}
\textsc{Pencheva}, G.~V. ; \textsc{Vohral{\'{\i}}k}, M. ; \textsc{Wheeler},
  M.~F.  ; \textsc{Wildey}, T.:
\newblock Robust a posteriori error control and adaptivity for multiscale,
  multinumerics, and mortar coupling.
\newblock {In: }\emph{SIAM J. Numer. Anal.} 51 (2013), Nr. 1, 526--554.
\newblock \url{http://dx.doi.org/10.1137/110839047}. --
\newblock DOI 10.1137/110839047. --
\newblock ISSN 0036--1429

\bibitem[47]{QV1999}
\textsc{Quarteroni}, A. ; \textsc{Valli}, A.:
\newblock \emph{Domain decomposition methods for partial differential
  equations}.
\newblock Oxford University Press, 1999 (CMCS-BOOK-2009-019)

\bibitem[48]{dunepymor}
\textsc{Rave}, S. ; \textsc{Schindler}, F.:
\newblock \emph{dune-pymor}.
\newblock \url{http://dx.doi.org/10.5281/zenodo.13896}

\bibitem[49]{dunegdt}
\textsc{Schindler}, F.:
\newblock \emph{{DUNE} generic discretization toolbox (dune-gdt)}.
\newblock \url{http://dx.doi.org/10.5281/zenodo.13893}

\bibitem[50]{dunegridmultiscale}
\textsc{Schindler}, F.:
\newblock \emph{dune-grid-multiscale}.
\newblock \url{http://dx.doi.org/10.5281/zenodo.13894}

\bibitem[51]{dunehdd}
\textsc{Schindler}, F.:
\newblock \emph{DUNE high-dimensional discretizations (dune-hdd)}.
\newblock \url{http://dx.doi.org/10.5281/zenodo.13895}

\bibitem[52]{Sme2015}
\textsc{Smetana}, K.:
\newblock A new certification framework for the port reduced static
  condensation reduced basis element method.
\newblock {In: }\emph{Comput. Methods Appl. Mech. Engrg.} 283 (2015), 352--383.
\newblock \url{http://dx.doi.org/10.1016/j.cma.2014.09.020}. --
\newblock DOI 10.1016/j.cma.2014.09.020. --
\newblock ISSN 0045--7825

\bibitem[53]{spe10}
\textsc{{Society of Petroleum Engineers}}:
\newblock \emph{2001 SPE Comparative Solution Project}.
\newblock \url{http://www.spe.org/web/csp/index.html}.
\newblock \,Version:\,2001

\bibitem[54]{Tan}
\textsc{Tantau}, T.:
\newblock \emph{The {T}i{\it k}{Z} and PGF Packages},
  \url{http://sourceforge.net/projects/pgf/}

\bibitem[55]{Tan2013}
\textsc{Tantau}, T.:
\newblock Graph drawing in {T}i{\it k}{Z}.
\newblock {In: }\emph{J. Graph Algorithms Appl.} 17 (2013), Nr. 4, 495--513.
\newblock \url{http://dx.doi.org/10.7155/jgaa.00301}. --
\newblock DOI 10.7155/jgaa.00301. --
\newblock ISSN 1526--1719

\bibitem[56]{VPP2003}
\textsc{Veroy}, K. ; \textsc{Prud'homme}, C.  ; \textsc{Patera}, A.~T.:
\newblock Reduced-basis approximation of the viscous {B}urgers equation:
  rigorous a posteriori error bounds.
\newblock {In: }\emph{C. R. Math. Acad. Sci. Paris} 337 (2003), Nr. 9,
  619--624.
\newblock \url{http://dx.doi.org/10.1016/j.crma.2003.09.023}. --
\newblock DOI 10.1016/j.crma.2003.09.023. --
\newblock ISSN 1631--073X

\bibitem[57]{Voh2007}
\textsc{Vohral{\'{\i}}k}, M:
\newblock A posteriori error estimates for lowest-order mixed finite element
  discretizations of convection-diffusion-reaction equations.
\newblock {In: }\emph{SIAM J. Numer. Anal.} 45 (2007), Nr. 4, 1570--1599
  (electronic).
\newblock \url{http://dx.doi.org/10.1137/060653184}. --
\newblock DOI 10.1137/060653184. --
\newblock ISSN 0036--1429

\bibitem[58]{Yan2014}
\textsc{Yano}, M.:
\newblock A minimum-residual mixed reduced basis method: Exact residual
  certification and simultaneous finite-element reduced-basis refinement.
\newblock {In: }\emph{preprint, MIT}  (2014).
\newblock
  \url{http://augustine.mit.edu/methodology/papers/yano_M2AN_Aug_2014.pdf}

\end{thebibliography}
\end{document}